\documentclass{amsart}

\DeclareFontFamily{U}{mathb}{\hyphenchar\font45}
\DeclareFontShape{U}{mathb}{m}{n}{
      <5> <6> <7> <8> <9> <10> gen * mathb
      <10.95> mathb10 <12> <14.4> <17.28> <20.74> <24.88> mathb12
      }{}
\DeclareSymbolFont{mathb}{U}{mathb}{m}{n}
\DeclareMathSymbol{\righttoleftarrow}{3}{mathb}{"FD}

\usepackage[english]{babel}
\usepackage[utf8]{inputenc}
\usepackage{graphicx}
\usepackage{xcolor}
\usepackage{longtable}
\usepackage{float}

\usepackage[all]{xy}

\usepackage{amsmath}
\usepackage{amsthm}
\usepackage{amssymb}
\usepackage{amscd}
\usepackage{tikz}
\usepackage{csquotes}
\usepackage{imakeidx} 
\usepackage{appendix}
\usepackage{tikz-cd}
\usepackage{verbatim}
\usepackage{enumitem}
\usepackage{multicol}
\usepackage{aliascnt}
\usepackage[backref=page]{hyperref}
\usepackage{cleveref}
\usepackage{mathtools}

\hypersetup{colorlinks=true,linkcolor=blue,anchorcolor=blue,citecolor=blue}

\newtheorem{thm}{Theorem}[section]
\crefname{thm}{Theorem}{Theorems} 

\newaliascnt{prop}{thm}       
\newtheorem{prop}[prop]{Proposition} 
\aliascntresetthe{prop}       
\crefname{prop}{Proposition}{Propositions} 

\newaliascnt{lemma}{thm}
\newtheorem{lemma}[lemma]{Lemma}
\aliascntresetthe{lemma}
\crefname{lemma}{Lemma}{Lemmas}

\newaliascnt{cor}{thm}
\newtheorem{cor}[cor]{Corollary}
\aliascntresetthe{cor}
\crefname{cor}{Corollary}{Corollaries}

\newaliascnt{conj}{thm}
\newtheorem{conj}[conj]{Conjecture}
\aliascntresetthe{conj}
\crefname{conj}{Conjecture}{Conjectures}

\newaliascnt{question}{thm}
\newtheorem{question}[question]{Question}
\aliascntresetthe{question}
\crefname{question}{Question}{Questions}

\newaliascnt{situation}{thm}
\newtheorem{situation}[situation]{Situation}
\aliascntresetthe{situation}
\crefname{situation}{Situation}{Suestions}

\theoremstyle{definition}
\newaliascnt{defin}{thm}
\newtheorem{defin}[defin]{Definition}
\aliascntresetthe{defin}
\crefname{defin}{Definition}{Definitions}

\newaliascnt{construction}{thm}
\newtheorem{construction}[construction]{Construction}
\aliascntresetthe{construction}
\crefname{construction}{Construction}{Constructions}
\Crefname{construction}{Construction}{Constructions}

\newaliascnt{claim}{thm}
\newtheorem{claim}[claim]{Claim}
\aliascntresetthe{claim}
\crefname{claim}{Claim}{Claims}
\Crefname{claim}{Claim}{Claims}

\newtheorem{example}[thm]{Example}
\newtheorem{exercise}[thm]{Exercise}

\theoremstyle{remark}
\newaliascnt{rmk}{thm}
\newtheorem{rmk}[rmk]{Remark}
\aliascntresetthe{rmk}
\crefname{rmk}{Remark}{Remarks}
\Crefname{rmk}{Remark}{Remarks}

\numberwithin{equation}{section}

\newcommand{\Q}{\mathbb Q}
\newcommand{\F}{\mathbb F}

\newcommand{\Z}{\mathbb Z}
\newcommand{\G}{\mathbb G}

\newcommand{\w}{\mathfrak w}
\renewcommand{\P}{\mathbb P}
\newcommand{\Spec}{\operatorname{Spec}}
\newcommand{\Br}{\operatorname{Br}}

\newcommand{\A}{\mathbb A}
\newcommand{\mc}[1]{\mathcal{#1}}
\newcommand{\cl}{\overline}

\renewcommand{\phi}{\varphi}

\newcommand{\on}[1]{\operatorname{#1}}
\newcommand{\ang}[1]{\langle{#1}\rangle}

\newcommand{\Pic}{\mathrm{Pic}}

\newcommand{\eqto}{\stackrel{\lower1.5pt\hbox{$\scriptstyle\sim\,$}}\to}
\newcommand{\eqdashto}{\stackrel{\lower1.5pt\hbox{$\scriptstyle\sim\,$}}\dashrightarrow}

\makeatletter
\@namedef{subjclassname@2020}{\textup{2020} Mathematics Subject Classification}
\makeatother

\definecolor{MyDarkGreen}{rgb}{0.0,0.5,0.0}

\newcommand{\Sel}{\mathrm{Sel}}
\renewcommand{\c}{\mathfrak c}

\newcommand{\X}{\mathcal X}
\newcommand{\rk}{\mathrm{rk}}
\newcommand{\e}{\varepsilon}
\newcommand{\Ga}{\Gamma}
\renewcommand{\O}{\mathcal O}
\newcommand{\Gal}{\mathrm{Gal}}

\newcommand{\val}{\mathrm{val}}
\DeclareFontEncoding{OT2}{}{}
\DeclareTextFontCommand{\textcyr}{\fontencoding{OT2}\fontfamily{wncyr}\fontseries{m}\fontshape{n}\selectfont}

\newcommand{\Sha}{\textcyr{Sh}}

\newcommand{\bthe}{\begin{thm}}
\newcommand{\ethe}{\end{thm}}
\newcommand{\ble}{\begin{lemma}}
\newcommand{\ele}{\end{lemma}}
\newcommand{\bpr}{\begin{prop}}
\newcommand{\epr}{\end{prop}}
\newcommand{\bco}{\begin{cor}}
\newcommand{\eco}{\end{cor}}
\newcommand{\bde}{\begin{defin}}
\newcommand{\ede}{\end{defin}}
\newcommand{\brem}{\begin{rmk}}
\newcommand{\erem}{\end{rmk}}
\newcommand{\bexe}{\begin{exercise}}
\newcommand{\eexe}{\end{exercise}}
\newcommand{\bexa}{\begin{example}}
\newcommand{\eexa}{\end{example}}
\newcommand{\bconj}{\begin{conj}}
\newcommand{\econj}{\end{conj}}
\newcommand{\bques}{\begin{question}}
\newcommand{\eques}{\end{question}}

\title{Zero-cycles on surfaces dominated by products of hyperelliptic curves}

\author{Jean-Louis Colliot-Th\'el\`ene}
\address{CNRS\\
    Laboratoire de Math\'ematiques d'Orsay\\
    Universit\'e Paris-Saclay\\
307 rue Michel Magat, 91400\\
    Orsay, France} 
\email{jean-louis.colliot-thelene@universite-paris-saclay.fr}

\author{Federico Scavia}
\address{CNRS\\
	Institut Galil\'ee\\
	Universit\'e Sorbonne Paris Nord\\
	99 avenue Jean-Baptiste Cl\'ement, 93430\\ 
	Villetaneuse, France}
\email{scavia@math.univ-paris13.fr}

\author{Alexei Skorobogatov}
\address{Department of Mathematics \\
     South Kensington Campus \\
     Imperial College London \\
SW7 2AZ United Kingdom}
\address{Institute for the Information Transmission Problems\\
Russian Academy of Sciences\\
Moscow, 127994 Russia}
\email{a.skorobogatov@imperial.ac.uk}


\date{\today}
\begin{document}

\begin{abstract}
Conditionally on the finiteness of the relevant Tate--Shafarevich groups, we prove a local-to-global result for zero-cycles of degree $1$ on the surfaces given by $y^2=f_1(x_1)f_2(x_2)$, where the 
polynomials $f_1$ and $f_2$ are algebraically general.
The proof combines the fibration method, parity results for $2$-Selmer groups in quadratic twist families, and a variant of a theorem of Morgan on the variation of the Cassels--Tate pairing.
\end{abstract}

\maketitle

\tableofcontents

\section{Introduction}
Let $X$ be a smooth, projective, geometrically integral variety over a
number field $k$. Assume that the adelic space
\[
        X({\bf A}_k)=\prod_v X(k_v)
\]
is non-empty. It has been known for a long time that this condition does
not suffice to ensure that $X(k)$ is non-empty. It does not even suffice
to ensure that the index $I(X)$ of $X$ is equal to $1$, that is, that
there exists a zero-cycle of degree $1$ on $X$, as shows the example of smooth projective curves of genus 1.

Most counterexamples known before 1970 were accounted for by Manin by
means of the Brauer--Grothendieck group $\Br(X)$. Namely, class field
theory gives a pairing  
$X({\bf A}_k)\times\Br(X)\to\Q/\Z$, and one has an inclusion
\[
        X(k) \subset X({\bf A}_k)^{\Br},
\]
where $X({\bf A}_k)^{\Br}$ denotes the Brauer--Manin set of adelic points
orthogonal to $\Br(X)$. 

\begin{question}\label{question:bmo}
    If $X({\bf A}_k)^{\Br}$ is non-empty, does it follow that $X(k)\neq \emptyset$?
\end{question}
 For arbitrary smooth
projective varieties, \Cref{question:bmo} has a negative answer \cite{S99}. It is conjectured that \Cref{question:bmo} has a positive answer in the following cases: for rationally connected varieties \cite[Conjecture 14.1.2]{CTS21}, for K3 surfaces \cite[Conjecture 14.3.12]{CTS21}, and for homogeneous spaces of abelian varieties, where Manin observed that a positive answer to \Cref{question:bmo} follows from the finiteness of the Tate--Shafarevich groups (see \cite[p.~379]{CTS21}).

Although \Cref{question:bmo} has a negative answer in general, it is still open whether the following weaker form holds for arbitrary smooth projective varieties.

\begin{conj}\label{conj}
    If $X({\bf A}_k)^{\Br}$ is non-empty, then $I(X)=1$.
\end{conj}

\brem{\rm
Conjecture \ref{conj} holds for a smooth, projective, geometrically integral curve $C$ if the Tate--Shafarevich group of the Jacobian $J=\Pic^0_{C/k}$ is finite. Indeed, the canonical embedding
$C\hookrightarrow\Pic^1_{C/k}$ gives an inclusion 
$C({\bf A}_{k})^{\Br}
\subset\Pic^1_{C/k}({\bf A}_{k})^{\Br}$, so the last set is non-empty. By Manin,
the finiteness of $\Sha(J)$
then implies that 
$\Pic^1_{C/k}$ has a $k$-point.
A well-known argument
(Lemma \ref{lemma1} below)
gives $I(C)=1$.}
\erem

In this paper we extend this observation to certain surfaces of general type dominated by products of two hyperelliptic curves.

\Cref{conj} follows from a more precise conjecture put forward by Colliot-Thélène--Sansuc  and Kato--Saito \cite[Conjecture 15.1.3]{CTS21}. Starting with
work of Salberger on conic bundle surfaces \cite{Sal88},
substantial progress has been made on \Cref{conj} for varieties
admitting a dominant morphism to $\P^1_k$ with rationally connected
fibres of a type already under control; see, for example,
work of Harpaz and Wittenberg \cite{HW16}.

In this introduction and in the main theorem of this paper
\emph{we
assume the finiteness of the Tate--Shafarevich groups of the relevant
abelian varieties.}

For a variety $X$ with a dominant morphism to $\P^1_k$ whose generic
fibre is a torsor under an abelian variety $A_t$ over $k(\P^1_k)=k(t)$,
Swinnerton-Dyer \cite{SD95} found a method to prove, under rather
stringent assumptions, that \Cref{question:bmo} has a positive answer. These
assumptions are due to the technical nature of the method and are
presumably too strong. In many cases where the method applies, only a very small subgroup of $\on{Br}(X)$ comes into play, and the
final result simply takes the form of the Hasse principle:
\[
        X({\bf A}_k)\neq \emptyset \quad \Longrightarrow \quad X(k)\neq \emptyset .
\]
The goal is to find a value $t_0\in k$ of the parameter
$t$ on $\P^1_k$ such that the fibre $X_{t_0}$ is smooth, has points in
all completions of $k$, and the relevant Selmer group of the abelian
variety $A_{t_0}$ is sufficiently small. Under the finiteness hypothesis
for $\Sha(A_{t_0})$, this then implies that
$X_{t_0}(k)\neq \emptyset$, and hence $X(k)\neq \emptyset$.

Swinnerton-Dyer's method was elaborated in \cite{CTSSD98} and further  developed in papers by Swinnerton-Dyer, Colliot-Thélène,  Swinnerton-Dyer--Skorobogatov, and Wittenberg.
The method was considerably simplified
by Harpaz and Skorobogatov in \cite{HS16}.
In \cite{Har19} Harpaz extended the method by introducing the second descent
and variation of the Cassels--Tate
pairing under quadratic twists.

In his recent paper \cite{Mor25}, Morgan proved the existence of quadratic twists that do not change the 2-Selmer group and give the
Cassels--Tate pairing on it any form
allowed by the constraints of
Poonen and Stoll \cite{PS}.
In the present paper, we use Morgan's new method 
as well as some of his earlier results. For another recent application of this technique
see \cite[Section 4]{MS}.

The authors of \cite{SSD05} and \cite{HS16} studied
\Cref{question:bmo} on
Kummer surfaces with
affine equation
\begin{equation}
        y^2=f_1(x_1)f_2(x_2), \label{eq1}
\end{equation}
where the polynomials $f_i$ have degree at most $4$. The
relevant family over $\P^1_k$ arises by considering the pair of simultaneous
quadratic-twist equations
\begin{equation}
    z_1^2=t f_1(x_1), \qquad 
    z_2^2=t f_2(x_2). \label{eq2}
\end{equation}

In the present paper we initiate the study of \Cref{conj} for smooth
projective surfaces with affine equation (\ref{eq1}),
where $f_1$ and $f_2$ have 
degree greater than 4. These are surfaces
of general type, for which  an affirmative answer to
\Cref{question:bmo} is not expected. More precisely, here is the setting that we will consider.

\begin{situation}\label{sit1}
Let $k$ be a number field. For $i=1,2$, let $g_i\geq 2$ be an integer. Let $f_i(x)\in k[x]$ be a separable polynomial of degree $2g_i+2$, discriminant $\Delta_i$ and splitting field $L_i$.
Assume that $\Gal(L_i/k)\simeq S_{2g_i+2}$ for $i=1,2$, and that
$L_1$ and $L_2$ are linearly disjoint over $k$. 

Let $S_0$ be a finite set of places of $k$ containing the archimedean places, the primes dividing $2$, the primes $v$ where at least one of $f_1(x)$ and $f_2(x)$ is not in $\O_v[x]$, or the leading coefficient of at least one of $f_1(x)$ and $f_2(x)$ is not in $\O_v^\times$, and the primes $v$ with residue field of size $<4\max\{g_1,g_2\}^2$. 
Assume the existence of primes $w_1\neq w_2$ not in $S_0$ such that each $f_i(x)$ has a
unique double root modulo $w_i$ and $\val_{w_i}(\Delta_i)=2$, and $f_i(x)$ is separable modulo $w_j$, for $i\neq j$.

Let $C_i$ be a smooth projective hyperelliptic curve with affine equation $y^2=f_i(x)$. Let $X$ be a smooth, projective, geometrically integral variety birational to the affine variety with equation \[y^2=f_1(x_1)f_2(x_2).\] 
\end{situation}

Let us denote by $\iota_1\colon C_1\to C_1$ and $\iota_2\colon C_2\to C_2$ the hyperelliptic involutions.
We may take $X$ to be the quotient of the blow-up of $C_1\times_k C_2$ at the fixed locus of the involution $(\iota_1,\iota_2)$, by the action of
the involution extending $(\iota_1,\iota_2)$. This is the minimal desingularization of the quotient of $C_1\times C_2$ by $(\iota_1,\iota_2)$.

Our main result is the following.

\begin{thm}\label{t1}
We place ourselves in \Cref{sit1}. Assume that the $2$-primary torsion subgroups of the Tate--Shafarevich groups of the Jacobians of all quadratic twists of $C_1$ and $C_2$ over $k$ are finite. Then $X({\bf A}_k)\neq\emptyset$ implies that $X$ has a zero-cycle of degree $1$.
\end{thm}

Several comments on \Cref{t1} are in order.

\begin{enumerate}
    \item The surfaces $X$ appearing in \Cref{sit1} are of general type (\Cref{lemma:X-general-type}).
    \item The set of $k$-points $(a_0,\dots,a_{2g_1+2};b_0,\dots,b_{2g_2+2})$ of $\mathbb{A}^{2g_1+3}_k\times_k\mathbb{A}^{2g_2+3}_k$ such that \Cref{sit1} holds for the polynomials $f_1(x)=a_{2g_1+2}x^{2g_1+2}+\dots+a_0$ and $f_2(x)=b_{2g_2+2}x^{2g_2+2}+\dots+b_0$ satisfies weak approximation. In fact, one can even fix suitable $w_1$ and $w_2$ in advance; see \Cref{rmk:wa}.
    \item If the degree of at least one $f_i$ is odd, then it is not difficult to show that any smooth projective model $X$ has a $k$-rational point.
    \item The assumption that
    the polynomials $f_i(x)$ are algebraically general (linear independence of their splitting fields and maximal Galois action) is used in the proof of \Cref{prop:fibration-method} to show that the vertical Brauer group of a certain fibration over $\P^1_k$ attached to $X$ comes from $\on{Br}(k)$. It is also used in the proofs of Propositions~\ref{lem:maximal-galois-image} and \ref{lem:maximal-galois-image-new}. 
    \item Local conditions similar to those imposed in \Cref{sit1}  appear also in the 
    aforementioned papers, for example in \cite[Theorems A, B]{HS16} and \cite[Theorem 1.2]{Mor25}. It would be interesting to understand to 
    what extent the combination of algebraic genericity of the polynomials with such local conditions implies that there is no Brauer-Manin obstruction to the existence of a zero-cycle of degree one.
    \end{enumerate}

Here is a brief outline of the paper.
In Section \ref{S2} we use the well-known fibration theorem for the
universal family of quadratic twists
to prove the existence of $t\in k^\times$ such that both curves in (\ref{eq2}) are everywhere locally soluble.
Section \ref{S3} summarizes various known results from the literature. In
Section \ref{S4} we prove that it is possible to choose $t\in k^\times$ such that both curves in (\ref{eq2}) are everywhere locally soluble and, moreover, the 2-Selmer groups of their Jacobians are $\F_2$-vector spaces of odd dimension.
In Section \ref{S5} we apply Morgan's method to prove the existence of $t\in k^\times$ such that, in addition to the above conditions, both Cassels--Tate
pairings have any given form allowed
by the theory of Poonen and Stoll.
The proof of Theorem \ref{t1} is
completed in Section \ref{S6}
by choosing $t\in k^\times$ so that
the kernel of each of the two
Cassels--Tate pairings is generated
by the class of $\Pic^1_{C_i/k}$. The 
existence of a zero-cycle of degree 1
on each twisted curve is an easy
consequence. Finally, in the appendix
we show that the surfaces featuring in \Cref{sit1} are of general type.

\subsubsection*{Notation}
For a field $k$, let $\cl{k}$ be a separable closure of $k$ and $\Gamma_k\coloneqq \on{Gal}(\cl{k}/k)$ the absolute Galois group of $k$.

For a smooth projective curve $C$ over a field $k$, we denote by $\Pic_{C/k}$ the Picard scheme of $C$ over $k$.

For an abelian variety $A$ over a 
field $k$ we denote by $A^\vee$
the dual abelian variety.
When $k$ is a number field, we denote by $\Sha(A)$ the Tate--Shafarevich group of $A$ and by $\Sha(A)_{\mathrm{nd}}$ the quotient of $\Sha(A)$ by its 
(conjecturally trivial) maximal divisible subgroup. For an integer $m\geq 1$, let $\Sel^m(A)$ be the $m$-Selmer group of $A$. Let $\rk_2(A)$ be the $2^\infty$-Selmer rank of $A$, defined as the
$\Z_2$-rank of $\varprojlim \Sel^{2^n}(A)$.
If the 2-primary subgroup of $\Sha(A)$ is finite, then  $\rk_2(A)=\rk(A(k))$.

\subsubsection*{Acknowledgement}
The authors are grateful to the Lodha Mathematical Sciences Institute in Mumbai for hospitality and support during the thematic programme on rational points, algebraic cycles and the local-global principle.
The third named author also thanks 
the Max Planck Institute for Mathematics in Bonn
for excellent working conditions and support. We thank Adam Morgan for a helpful discussion and his
comments on the first draft of the paper. We originally proved 
Theorem \ref{t1} under the additional assumption that $g_1$ and $g_2$ are even. The extension to the case
where $g_1$ and $g_2$ have arbitrary parities was explained to us by
Adam Morgan.

\section{The fibration argument}
\label{S2}

Let $k$ be a number field. 
Let $C_i$ be a smooth projective hyperelliptic curve with affine equation $y^2=f_i(x)$.
For $t\in k^\times$ we write $C_i^t$ for the quadratic twist of $C_i$ by $t$. This is the 
hyperelliptic curve with affine equation $y^2=tf_i(x)$. 
Define $\mathfrak S$ as the set of $t\in k^\times$ such that $C_i^t({\bf A}_k)\neq\emptyset$ for $i=1,2$. 

\begin{prop}\label{prop:fibration-method}
Let $k$ be a number field. For $i=1,2$, let $g_i$ be a positive integer. Let $f_i(x)\in k[x]$ be a separable polynomial of degree $2g_i+2$, discriminant $\Delta_i$ and splitting field $L_i$.
Assume that $\Gal(L_i/k)\simeq S_{2g_i+2}$ for $i=1,2$, and
$L_1$ and $L_2$ are linearly disjoint over $k$. 
Let $C_i$ be a smooth projective hyperelliptic curve with affine equation $y^2=f_i(x)$.
Let $X$ be a smooth, projective, geometrically integral variety birational to the affine variety with equation $y^2=f_1(x_1)f_2(x_2)$.
Suppose that $X({\bf A}_k)\neq\emptyset$.
    
    Let $\Sigma$ be a finite set of places of $k$, and for every $v\in \Sigma$ let $t_v\in k_v^\times$ be such that $C_i^{t_v}(k_v)\neq\emptyset$ for $i=1,2$. Then, for every real number $\epsilon>0$, there exists $t\in \mathfrak S$ such that $|t-t_v|_v<\epsilon$ in $k_v$ for all $v\in \Sigma$. In particular, $\mathfrak S$ is non-empty. 
\end{prop}

\begin{proof}
The second assertion follows from the first by taking $\Sigma=\emptyset$. It remains to prove the first assertion. 

 The fixed locus $Z$ of the  involution $(\iota_1,\iota_2)$ on $C_1\times_kC_2$ is a finite \'etale closed subscheme of degree $(2g_1+2)(2g_2+2)$ over $k$. Define $Y$ as the blow-up of $C_1\times_k C_2$ at $Z$. Let $\X$ denote the quotient of $Y\times_k \G_{m,k}$ by the diagonal action of $\mu_2$ (acting by the extension of $(\iota_1,\iota_2)$ on the first factor, and by $x\mapsto -x$ on the second factor). The fibre of $\X$ over $t\in \G_m(k)$ is 
the quadratic twist of $Y$ by $t$. Let $\X\subset\tilde\X$ be a smooth compactification fitting into a commutative diagram
\[
\begin{tikzcd}
\X \arrow[r] \arrow[d] & \tilde{\X} \arrow[d] \\
\mathbb{G}_{m,k} \arrow[r] & \mathbb{P}^1_k
\end{tikzcd}
\]
The generic fibre of the projection $\mathcal{X}\to X$ is a $\mathbb{G}_m$-torsor, and hence in particular the threefold $\tilde\X$ is stably birational to $X$. It is clear that $X$ is the minimal desingularization of
$(C_1^t\times C_2^t)/(\iota_1, \iota_2)$ for any $t\in k^\times$,
where, by a slight abuse of notation, we denote by $\iota_i\colon C_i^t\to C_i^t$ the respective hyperelliptic involutions.

For $i=1,2$, let $l_i\coloneqq k[x]/(f_i(x))$, and fix an inclusion $l_i\subset L_i$. 

\begin{claim}\label{claim:vertical-brauer}
    The vertical Brauer group of $\tilde\X$ over $\P^1_k$ is the image of $\Br(k)$ in $\Br(\tilde\X)$.
\end{claim}

\begin{proof}[Proof of \Cref{claim:vertical-brauer}]
This will follow from \cite[Proposition 11.5.2(i)]{CTS21} once we have shown that the natural homomorphism $k^\times/k^{\times 2}\to (l_1l_2)^\times/(l_1l_2)^{\times 2}$ is injective.

Since $L_1$ and $L_2$ are linearly disjoint over $k$, we have
\[\Gal(L_1L_2/k)\simeq \Gal(L_1/k)\times\Gal(L_2/k)\simeq S_{2g_1+2}\times S_{2g_2+2}.\] 
The subfield $l_1l_2\subset L_1L_2$ is the fixed
field of $H=S_{2g_1+1}\times S_{2g_2+1}$, where $S_{2g_i+1}$ is the
stabilizer of a root of $f_i$.
The quadratic extensions of $k$ contained in $L_1L_2$ are 
$k(\sqrt{\Delta_1})$, $k(\sqrt{\Delta_2})$ and $k(\sqrt{\Delta_1\Delta_2})$,
where $\Delta_i\in k^\times$ is the discriminant of $f_i$ for $i=1,2$. 
It suffices to show that none of these
quadratic extensions of $k$ is contained in $l_1l_2$. 
Equivalently, it is enough to show 
that 
$H$ is not contained in the kernel of the corresponding quadratic
character of $S_{2g_1+2}\times S_{2g_2+2}$. This is indeed impossible, because $2g_i+2\geq 4$ and hence each
$S_{2g_i+1}$ contains odd permutations.
\end{proof}

Since $X$ is stably birational to $\X$, solubility of $X$ everywhere locally is equivalent to the solubility of $\X$ everywhere locally. By \Cref{claim:vertical-brauer}, we may apply the fibration theorem for the universal family of quadratic twists
\cite[Theorem A]{CTS00} (see also \cite[Corollary 14.2.19]{CTS21}) to find a $t\in\mathfrak S$ arbitrarily close to $t_v$ for $v\in \Sigma$.
\end{proof}

\section{Preliminaries on Jacobians of hyperelliptic curves}
\label{S3}

\subsection{Theta characteristics and the Poonen--Stoll class}

\begin{defin}\label{defin:classes}
Let $C$ be a smooth, projective, geometrically integral curve over a field $k$, of
positive genus $g$. Let $J=\operatorname{Pic}^0_{C/k}$ be the Jacobian of $C$, and let
$\lambda$ be the canonical principal polarization of $J$.

\begin{enumerate}
    \item Let $\mathfrak c\coloneqq [\Theta_C]\in H^1(k,J[2])$, where $\Theta_C\subset \on{Pic}^{g-1}(C_{\bar k})$ is the $J[2]$-torsor defined by $\Theta_C\coloneqq
\{L\in \operatorname{Pic}^{g-1}(C_{\bar k}): L^{\otimes 2}\simeq \omega_C\}$; see \cite[Definition~3.8]{PR11}.
\item Let $\mathfrak c_\lambda\in H^1(k,J^\vee[2])$ be the class of the $J^\vee[2]$-torsor  of symmetric line bundles on $J$ inducing the polarization $\lambda$; see \cite[(17)]{PR11}. 
\item Let $c_\lambda\in H^1(k,J^\vee)$ be the class associated to $\lambda$ by Poonen--Stoll \cite[\S~4]{PS}.
\item Let $c\in H^1(k,J)$ be the class of the $J$-torsor $\operatorname{Pic}^{g-1}_{C/k}$.
\end{enumerate}
\end{defin}

\begin{prop}
Let $C$ be a smooth, projective, geometrically integral curve over a number field $k$, of
positive genus $g$. Let $J=\operatorname{Pic}^0_{C/k}$ be the Jacobian of $C$, and let $\lambda\colon J\to J^\vee$ be the canonical principal polarization. Then 
    \[c_\lambda\in \Sha(J^\vee)[2],\quad c\in \Sha(J)[2],\quad \mathfrak{c}_\lambda\in \on{Sel}^2(J^\vee/k),\quad \mathfrak c\in \on{Sel}^2(J/k).\] Moreover, the four classes fit into the following commutative diagram:
\begin{equation}\label{eq:four-classes}
\begin{tikzcd}
\on{Sel}^2(J/k) \arrow[r,"\lambda_*"] \arrow[d] &
\on{Sel}^2(J^\vee/k) \arrow[d] \\
\Sha(J)[2] \arrow[r,"\lambda_*"] &
\Sha(J^\vee)[2],
\end{tikzcd}
\qquad
\begin{tikzcd}
\mathfrak c \arrow[r,mapsto,"\lambda_*"] \arrow[d,mapsto] &
\mathfrak c_\lambda \arrow[d,mapsto] \\
c \arrow[r,mapsto,"\lambda_*"] &
c_\lambda .
\end{tikzcd}
\end{equation}
\end{prop}

\begin{proof}
    We have $\mathfrak c_\lambda=\lambda_*(\mathfrak c)$ by \cite[Theorem~3.9]{PR11}, and $c_\lambda=\lambda_*(c)$ by \cite[Corollary~4]{PS}. By \cite[Remark~3.3]{PR11}, the natural map $H^1(k,J^\vee[2])\to H^1(k,J^\vee)$ sends $\mathfrak c_\lambda$ to $c_\lambda$. It follows that the natural map $H^1(k,J[2])\to H^1(k,J)$ sends $\mathfrak c$ to $c$. 
    Finally, by \cite[Corollary~2]{PS} we have $c_\lambda\in \Sha(J^\vee)[2]$, and the other assertions follow from the commutativity of \eqref{eq:four-classes}.
\end{proof}

For an abelian variety $A$ over a number field $k$ we have the Cassels--Tate pairing
$$\Sha(A)\times\Sha(A^\vee)\to\Q/\Z,$$
see \cite[Proposition I.6.9]{ADT}.
When $J$ is the Jacobian of a curve with its canonical principal polarization
$\lambda\colon J\xrightarrow{\sim} J^\vee$, we denote by $\langle x,y\rangle_J$, for $x,y\in\Sha(J)$, the Cassels--Tate pairing of $x$ and $\lambda_*(y)$.

\begin{prop}\label{prop:ps-class}
Let $k$ be a number field, let $C:y^2=f(x)$ be a hyperelliptic curve of  
genus $g\geq 2$, where $f(x)\in k[x]$ is separable of degree $2g+2$, and let $J$ be
the Jacobian of $C$. Let $\mathcal R$ be the set of roots of $f$ in $\bar k$. Let $\mathfrak c\in \Sel^2(J/k)$ be the class of \eqref{eq:four-classes}.

\begin{enumerate}
\item There is a short exact sequence of $\Gamma_k$-modules
\[
0\longrightarrow J[2]\longrightarrow
\F_2^{\mathcal R}/\langle {\mathcal R}\rangle
\longrightarrow \F_2\longrightarrow 0,
\]
where the third map sends a subset of $\mathcal R$ to its
cardinality modulo $2$.
Let $\w$ be the image of
$1\in H^0(k,\F_2)$ under the connecting homomorphism $H^0(k,\F_2)\to H^1(k,J[2])$. 
\item If $g$ is even, then $\c=\w$.
If $g$ is odd, then $\c=0$.
\item The splitting field of $f$ over $k$ is equal to $k(J[2])$.
\item The natural map $H^1(k,J[2])\to H^1(k,J)$ sends $\w$ to $[\operatorname{Pic}^1_{C/k}]$.
\item After identifying $\frac12\Z/\Z$ with $\F_2$, we have
\[
\langle \mathfrak c,\mathfrak c\rangle_J \equiv
\dim_{\F_2}\Sha(J)_{\rm nd}[2]
\equiv
\dim_{\F_2}\Sel^2(J/k)+\rk_2(J)+\dim_{\F_2}J(k)[2]
\pmod 2.
\]
\item After identifying $\frac12\Z/\Z$ with $\F_2$, we have $\langle \mathfrak c,\mathfrak c\rangle_J \equiv N\pmod 2$, where $N$ is the number of places $v$ of $k$ such that $\on{Pic}^{g-1}(C_{k_v})=\emptyset$. In particular, if $C(k_v)\neq\emptyset$ for every place $v$ of $k$, then $\langle \mathfrak c,\mathfrak c\rangle_J=0$ in $\F_2$.
\item For every $t\in k^\times$, let $C^t:y^2=t f(x)$ be the quadratic twist
of $C$, let $J^t$ be its Jacobian, and let $\mathfrak c_t\in \Sel^2(J^t/k)\subset H^1(k,J^t[2])$ be the class of \eqref{eq:four-classes} for $J^t$. Under the canonical identification
$J^t[2]\simeq J[2]$, the class $\mathfrak c_t$ is identified with $\mathfrak c$.
\end{enumerate}
\end{prop}
\begin{proof}
\textup{(1)} See \cite[(8.3)]{Mor25}.

\textup{(2)} See \cite[Proposition~8.5]{Mor25}.

\textup{(3)} The sequence of (1) shows that $k(J[2])$ is contained in the splitting field of $f$. Conversely, since $2g+2\geq 6$, the natural action of the symmetric group $\on{Sym}(\mc{R})$ on $\on{Ker}(\F_2^{\mathcal R}/\langle {\mathcal R}\rangle\to \F_2)$ is faithful. Thus, by (1), any element of $\Gamma_k$ acting trivially on $J[2]$ acts trivially on $\mathcal R$. Therefore $k(J[2])$ contains the splitting field of $f$.

\textup{(4)} Let $\eta\in \operatorname{Pic}^2_{C/k}(k)$ be the class of a fibre of the hyperelliptic map $C\to\mathbb P^1_k$. 
Then $\w$ is the class of the
torsor under $J[2]$ which is the
preimage of $\eta$ under multiplication by 2 on $\Pic_{C/k}$.
Since $\deg(\eta)=2$, this torsor
is naturally a closed subset of 
$\Pic^1_{C/k}$, and this embedding
is compatible with the natural
injection $J[2]\to J$.
Thus the image of $\w$ in $H^1(k,J)$
is $[\Pic^1_{C/k}]$.

\textup{(5)} By \cite[Theorem~8]{PS}, we have $\langle \mathfrak c,\mathfrak c\rangle_J
\equiv
\dim_{\F_2}\Sha(J)_{\rm nd}[2]$, where $\Sha(J)_{\rm nd}$ is the quotient of $\Sha(J)$ by its maximal divisible subgroup. The conclusion follows 
from the exact sequence
\[
0\longrightarrow J(k)/2J(k)\longrightarrow \Sel^2(J/k)\longrightarrow \Sha(J)[2]\longrightarrow 0
\]
and the definition of $\rk_2(J)$.

\textup{(6)} See \cite[Corollary~12]{PS}.

\textup{(7)} The set of roots of $tf(x)$ coincides with the set of roots $\mc{R}$ of $f(x)$. We obtain a commutative diagram
\[
\begin{tikzcd}
0 \arrow[r] &
J[2] \arrow[r] \arrow[d, "\wr"] &
\F_2^{\mathcal R}/\langle \mathcal R\rangle
\arrow[r] \arrow[d, equal] &
\F_2 \arrow[r] \arrow[d, equal] &
0
\\
0 \arrow[r] &
J^t[2] \arrow[r] &
\F_2^{\mathcal R}/\langle \mathcal R\rangle
\arrow[r] &
\F_2 \arrow[r] &
0
\end{tikzcd}
\]
where the rows are the short exact sequences of \textup{(1)} for $J$ and $J^t$, and where the left vertical arrow is the unique map making the left square commute. The fact that the isomorphism $H^1(k,J[2])\to H^1(k,J^t[2])$ sends $\mathfrak c$ to $\mathfrak c_t$ now follows from (2).
\end{proof}

\subsection{The case of maximal Galois image}

\begin{prop}\label{lem:maximal-galois-image}
Let $C:y^2=f(x)$ be a hyperelliptic curve over a number field $k$ of
 genus $g\geq 2$, where $f\in k[x]$ is separable of degree
$2g+2$, and let $J$ be the Jacobian of $C$. Suppose that $G\coloneqq \operatorname{Gal}(f)\simeq S_{2g+2}$. 

\begin{enumerate}
\item We have $J(k)[2]=0$.
\item The $G$-module $J[2]$ is simple, and
$\operatorname{End}_G(J[2])=\F_2$.
\item We have $H^1(G,J[2])\simeq \F_2$, and the inflation map $H^1(G,J[2])\to H^1(k,J[2])$ sends the non-zero element to the class $\w$ of Proposition
\ref{prop:ps-class}(1). Consequently, if 
$C(\mathbf A_k)\neq\emptyset$, then
\[\Sel^2(J/k)\cap H^1(G,J[2])=\langle\w\rangle\quad\text{inside } \ H^1(k,J[2]).\]

\end{enumerate}
\end{prop}

\begin{proof}
(1) Immediate from \Cref{prop:ps-class}(1).

(2) See \cite[Lemma 8.2]{Mor25}.

(3) The first statement is \cite[Proposition 8.4]{Mor25}. By Proposition
\ref{prop:ps-class}(4), if 
$C(\mathbf A_k)\neq\emptyset$, then $\w\in\Sel^2(J/k)$.
\end{proof}

\begin{prop}\label{lem:maximal-galois-image-new}
Let $C:y^2=f(x)$ be a hyperelliptic curve over a number field $k$ of
genus $g\geq 2$, where $f\in k[x]$ is separable of degree
$2g+2$, and let $J$ be the Jacobian of $C$. Suppose that $G\coloneqq \operatorname{Gal}(f)\simeq S_{2g+2}$. 

\begin{enumerate}
\item For every cycle $h\in G\simeq S_{2g+2}$ of length $2g+1$, we have $J[2]^h=0$. 
    \item If $v$ is a prime of good reduction for $C$, unramified in the splitting field of $f$, and if $\on{Frob}_v$ acts as a cycle of length $2g+1$ on the roots of $f$, then $f$ has a root in $k_v$, and consequently every quadratic twist of $C$ has a $k_v$-point.
\end{enumerate}
\end{prop}

\begin{proof}
    (1) We use the identification of $J[2]$ given in \Cref{prop:ps-class}(1). The $h$-fixed subspace $(\F_2^{\mathcal R})^h$ has basis $e_1,e_2$, where $e_1$ is the fixed point and $e_2$ is the sum of the elements in the cycle of length $2g+1$. The  coordinate sum of $a_1 e_1+a_2 e_2$ is $a_1+(2g+1)a_2=a_1+a_2$. Hence the 
    intersection of $(\F_2^{\mathcal R})^h$ and the kernel of the augmentation map $\F_2^{\mathcal R}\to\F_2$ is generated by $e_1+e_2=\mathcal R$. Since $2g+1$ is odd, we have
    $H^1(\Z/(2g+1),\Z/2)=0$, which implies that $J[2]^h=0$.

(2) Since $\on{Frob}_v$ acts as a cycle of length $2g+1$, it fixes one root of the reduction of $f$ modulo $v$. Since $v$ is a place of good reduction, this root is simple. By Hensel's lemma, we may lift this root to a root $\alpha\in k_v$ of $f$ over $k_v$. For any quadratic twist $C^a:y^2=a f(x)$, the point $(\alpha,0)$ belongs to $C^a(k_v)$.
\end{proof}

\begin{cor}\label{cor:ps-class-twists}
Let $k$ be a number field, let $C:y^2=f(x)$ be a hyperelliptic curve of
genus $g\geq 2$, where $f(x)\in k[x]$ is separable of degree
$2g+2$, and let $J$ be the Jacobian of $C$. Suppose that
$\operatorname{Gal}(f)\simeq S_{2g+2}$. Let
$\mathfrak c\in \Sel^2(J/k)$ be the class of \eqref{eq:four-classes}.
For every $t\in k^\times$, viewing $\mathfrak c$ as an element of
$\Sel^2(J^t/k)$ via the canonical identification $J^t[2]\simeq J[2]$,
we have
\[
\langle \mathfrak c,\mathfrak c\rangle_{J^t}
\equiv
\dim_{\F_2}\Sel^2(J^t/k)+\rk_2(J^t)
\pmod 2.
\]
In particular, if $C^t(k_v)\neq\emptyset$ for every place $v$ of $k$,
then
\[
\dim_{\F_2}\Sel^2(J^t/k)+\rk_2(J^t)\equiv 0\pmod 2.
\]
\end{cor}

\begin{proof}
By \Cref{prop:ps-class}\textup{(7)}, the class $\mathfrak c$ is identified
with the Poonen--Stoll class of $J^t$. Applying
\Cref{prop:ps-class}\textup{(5)} to $C^t$, we obtain
\[
\langle \mathfrak c,\mathfrak c\rangle_{J^t}
\equiv
\dim_{\F_2}\Sel^2(J^t/k)+\rk_2(J^t)
+\dim_{\F_2}J^t(k)[2]
\pmod 2.
\]
By \Cref{lem:maximal-galois-image}(1), $J^t(k)[2]=0$, and the desired
congruence follows.

If $C^t(k_v)\neq\emptyset$ for every place $v$, then
$\operatorname{Pic}^{g-1}(C^t_{k_v})\neq\emptyset$ for every $v$.
Hence $\langle \mathfrak c,\mathfrak c\rangle_{J^t}=0$ by
\Cref{prop:ps-class}\textup{(6)}, giving the final assertion.
\end{proof}

\subsection{The minimal regular model}

\begin{lemma}\label{lemma:one-node-model} Let $K$ be a non-archimedean local field of odd residue characteristic, and let $C/K$ be the hyperelliptic curve $y^2=f(x)$, where $f\in\mathcal O_K[x]$ is separable of degree $2g+2$ and has leading coefficient in $\mathcal O_K^\times$. Suppose that the reduction $\overline f$ has exactly one double root and no other multiple root, and that $\operatorname{val}(\Delta(f))=2$. Then the minimal proper regular model of $C$ over $\mathcal O_K$ has special fibre consisting of two smooth geometrically integral irreducible components geometrically meeting in two ordinary double points. \end{lemma}

\begin{proof}
This is a special case of results of \cite{Liu96}; we give a direct proof. Let $\mathbb F$ be the residue field of $K$, and let $\pi$ be a uniformizer.
Let
\[
F(u)=u^{2g+2}f(1/u)\in \mathcal O_K[u].
\]
Consider the flat proper model $\mathcal C_0$ of $C$ over $\mathcal O_K$ obtained by gluing
the two affine schemes
\[
U_0\coloneqq \operatorname{Spec}\mathcal O_K[x,y]/(y^2-f(x)),\qquad U_\infty\coloneqq \operatorname{Spec}\mathcal O_K[u,z]/(z^2-F(u)).
\]
On $U_0$, the special fibre is singular exactly at the multiple roots of $\overline f$.
By assumption, there is only one such root, and it is a double root. On $U_\infty$, the points at infinity are given by the equations $u=0$, $z^2=\overline a_{2g+2}$, 
where $a_{2g+2}$ is the leading coefficient of $f$. Since $a_{2g+2}\in
\mathcal O_K^\times$ and $\on{char}(\F)$ is odd, these points are
smooth, because the partial derivative with respect to $z$ is nonzero.

Let $\alpha\in\mathbb F$ be the unique double root of $\overline f$, and let
$P$ be the corresponding singular point of $\mathcal C_{0,\mathbb F}$. After
translating $x$, we may assume that $\alpha=0$, so that $P=(x,y,\pi)$. We have $\widehat{O}_{\mathcal{C}_0,P}=\mathcal{O}_K[\![x,y]\!]/(y^2-f(x))$.

By Hensel's lemma, we may write $f(x)=q(x)h(x)$, where $q(x)\in\mathcal O_K[x]$ is a monic polynomial of degree $2$ with reduction $x^2$, and $h(x)\in\mathcal O_K[x]$ satisfies $h(0)\in\mathcal O_K^\times$. The assumption
$\operatorname{val}(\Delta(f))=2$ implies $\operatorname{val}(\Delta(q))=2$. Since $2\in \mathcal O_K^\times$, again by Hensel's lemma we may write $h(x)=h(0)s(x)^2$ for some  
$s(x)\in \mathcal O_K[\![x]\!]$
such that $s(0)\in \mathcal O_K^\times$. Thus $\widehat{O}_{\mathcal{C}_0,P}=\mathcal{O}_K[\![x,y']\!]/((y')^2-h(0)q(x))$, where $y'=y/s(x)$. Write $q(x)=x^2+bx+c$. Since $2\in \mathcal O_K^\times$, we have
\[\widehat{O}_{\mathcal{C}_0,P}=\mathcal{O}_K[\![x',y']\!]/((y')^2-h(0)(x')^2+\mu\pi^2),\] 
where $x'=x+b/2$, and where $\mu\in \mathcal O_K^\times$ because $\val(\Delta(q))=2$. Therefore, after passing to an \'etale cover given by taking the square root of $h(0)$, the completion of $\mathcal{C}_0$ at $P$ is given by 
$\mathcal{O}_K[\![x'',y'']\!]/(x''y''-\pi^2)$.
A simple calculation now shows that the blowup $\mathcal{C}$ of $\mathcal{C}_0$ at $P$ is regular. 
The special fibre of $\mathcal C$ is the union of the strict transform of the special fibre of $\mathcal{C}_0$ and the exceptional divisor, which meet in the points corresponding to the two tangent directions of the node (defined over the quadratic extension of $\F$ given by the square root of $\cl{h}(0)$).
Thus $\mathcal{C}$ is the required regular minimal model.
\end{proof}

\subsection{A Weil estimate}

\begin{lemma}\label{lem:weil-bound}
Let $g\geq 1$, let $q$ be an odd prime power, and let
$F(x)\in\mathbb F_q[x]$ have degree $2g+2$.
Let $a\in\mathbb F_q^\times$, and let $Y_a$ be the projective curve over
$\mathbb F_q$ with affine equation $y^2=aF(x)$. Assume that $F$ is separable,
or that $F$ has exactly one double root and no other multiple root. If
$q\geq 4g^2$, then $Y_a$ has a smooth $\mathbb F_q$-point.
\end{lemma}

\begin{proof}
If $F$ is separable, then $Y_a$ is a smooth curve of genus $g$. By the Weil
bound,
\[
|Y_a(\mathbb F_q)|\geq q+1-2g\sqrt q\geq 1.
\]
Suppose now that $F(x)=(x-\alpha)^2G(x)$ has exactly one double root
$\alpha$ and no other multiple root. Since this double root is unique, we have
$\alpha\in\mathbb F_q$. The normalization of $Y_a$ is the smooth projective
curve $\widetilde{Y}_a$ with affine model $z^2=aG(x)$, and it has genus $g-1$ because $\deg G=2g$. By the Weil bound,
\[
|\widetilde Y_a(\mathbb F_q)|
\geq q+1-2(g-1)\sqrt q .
\]
The normalization map is an isomorphism away from the singular point of
$Y_a$, and the inverse image of this singular point contains at most two
$\mathbb F_q$-points. Thus it is enough to show that
\[
q+1-2(g-1)\sqrt q>2.
\]
Since $q\geq 4g^2$, we have $\sqrt q\geq 2g$, and hence
\[
q+1-2(g-1)\sqrt q
\geq 2g\sqrt q+1-2(g-1)\sqrt q
=1+2\sqrt q>2.\qedhere
\]
\end{proof}

\section{Parity of Selmer groups in two families}

\label{S4}

\begin{lemma}\label{lemma:local-norm-index}
Let $K$ be a non-archimedean local field of odd residue characteristic,
with ring of integers $\mathcal O_K$. Let $K'/K$ be the unramified quadratic extension. Let $C: y^2=f(x)$
be a hyperelliptic curve of genus $g\geq 2$ 
over $K$, where $f\in \mathcal O_K[x]$ is separable of degree $2g+2$ with leading coefficient in $\mathcal O_K^\times$. Let $J$ be the Jacobian of~$C$.

\begin{enumerate}
\item Suppose that $J$ has good reduction over $K$. Then we have
\[
J(K)/N_{K'/K}J(K')=0,
\]
where $N_{K'/K}\colon J(K')\to J(K)$ is the norm map.
\item Suppose that the reduction of $f$ modulo the maximal ideal of $\mathcal O_K$ has exactly one double root
and no other multiple root, and that $\operatorname{val}(\Delta(f))=2$.
Then we have \[J(K)/N_{K'/K}J(K')\simeq \mathbb Z/2.\]
\end{enumerate}
\end{lemma}

\begin{proof}
Let $\mathbb F$ and $\mathbb F'$ be the residue fields of $K$ and
$K'$, respectively. Let $\mathcal J$ be the Néron model of $J$ over
$\mathcal O_K$, let $\mathcal J^0$ be its identity component, and let
$\Phi$ be the group scheme of connected components of the special fibre.

By \cite[Propositions 4.2 and 4.3]{Maz72}, we have
\begin{equation}\label{eq:j-modulo-norms}
J(K)/N_{K'/K}J(K')
\simeq H^1(\on{Gal}(K'/K),J(K'))\simeq H^1(\on{Gal}(\mathbb F'/\mathbb F),\Phi(\mathbb F')).
\end{equation}

(1) If $J$ has good reduction, then its Néron model is an
abelian scheme over $\mathcal O_K$. In particular, the special fibre is
connected, so $\Phi=0$, and we conclude by \eqref{eq:j-modulo-norms}.

(2) By Lemma \ref{lemma:one-node-model}, the special fibre of the minimal proper
regular model of $C$ consists of two irreducible components meeting in two
ordinary double points. Thus its dual graph has two vertices joined by two
edges. Therefore, by \cite[Chapter 9, \S 6, Proposition 10]{BLR90} we have $\Phi_{\cl{\mathbb F}}\simeq \mathbb Z/2$, and hence $\Phi\simeq \mathbb Z/2$. Thus $\Phi(\mathbb F')\simeq \Z/2$ and we conclude again by \eqref{eq:j-modulo-norms}.
\end{proof}

For $i=1,2$ and $t\in k^\times$, let $J_i^t=\Pic^0_{C_i^t/k}$ be the Jacobian of $C_i^t$. We use the canonical identification $J_i^t[2]\simeq J_i[2]$
of Galois modules.

\begin{lemma}\label{lemma4}
We place ourselves in \Cref{sit1}. Let $i,j\in\{1,2\}$ with
$i\neq j$.

\begin{enumerate}
\item Let $t\in k_{w_j}^\times$ be a unit. Then
\[
J_i(k_{w_j})/N J_i(k_{w_j}(\sqrt t))=0.
\]
\item Let $t\in k_{w_i}^\times$ be a unit. Then
\[
\dim_{\F_2}(J_i(k_{w_i})/N J_i(k_{w_i}(\sqrt t)))
=
\begin{cases}
0, & \text{if } t\in k_{w_i}^{\times 2},\\
1, & \text{if } t\notin k_{w_i}^{\times 2}.
\end{cases}
\]
\end{enumerate}
\end{lemma}

\begin{proof}
(1) Let $K=k_{w_j}$. The assumptions imply that $C_i$ has good reduction over $K$, and therefore $J_i$ has good reduction over $K$.

If $t\in K^{\times 2}$, then $K(\sqrt t)=K$ and the norm quotient is zero.
If $t\notin K^{\times 2}$, then since the residue characteristic of $K$ is odd, $K(\sqrt t)/K$ is the unramified quadratic
extension, and the conclusion follows from Lemma~\ref{lemma:local-norm-index}(1).

(2) Let $K=k_{w_i}$. If $t\in K^{\times 2}$, then
$K(\sqrt t)=K$, so the norm map is the identity map on $J_i(K)$. Hence the
quotient is zero. Suppose now that $t\notin K^{\times 2}$. By assumption, the residue characteristic of $K$ is odd, so that $K(\sqrt t)/K$ is the
unramified quadratic extension.
The conclusion follows from Lemma~\ref{lemma:local-norm-index}(2).
\end{proof}

Recall that $\mathfrak S$ is the set of $t\in k^\times$ such that $C_i^t({\bf A}_k)\neq\emptyset$ for $i=1,2$. 

\begin{prop}\label{prop:odd-odd}
We place ourselves in \Cref{sit1}. Assume that
$X({\bf A}_k)\neq\emptyset$. Then there exists $t\in\mathfrak S$ such
that $\dim_{\F_2}\Sel^2(J_1^t)$ and
$\dim_{\F_2}\Sel^2(J_2^t)$ are odd.
\end{prop}

\begin{proof} 
Let $t\in k^\times$, and let $L=k(\sqrt{t})$. By \cite[Theorem 10.12]{Mor19} we have
\begin{equation}
\dim_{\F_2}\Sel^2(J_i^t)\equiv\dim_{\F_2}\Sel^2(J_i)+
\sum_v\dim_{\F_2}\big(J_i(k_v)/N(J_i(L_w))\big)\bmod 2, \label{parity}
\end{equation}
where the sum is over all places of $k$, and $w$ is a place of $L$ over $v$. (If $v$ is split in $L$, then the summand is zero.)

Let $S$ be the union of $S_0$ and the primes dividing $\Delta_1\Delta_2$.

Assume $v\notin S$. Then $J_i$ has good reduction at $v$, so the  local root number is $w(J_i/L_w)=1$.
By  \cite[Theorem 1.8]{Mor23}, for such $v$, we thus have 
$$(\Delta_i,t)_{k_v}=(-1)^{\e(C_i/k_v)+\e(C_i^t/k_v)+\dim_{\F_2}\big(J_i(k_v)/N(J_i(L_w))\big)},$$
where $\e(C_i/k_v)=0$ if $C_i$ has a $k_v$-rational
divisor of degree $g-1$, otherwise $\e(C_i/k_v)=1$. Since $v\notin S$, the curve $C_i$ has good reduction at $v$. 
Then the special fibre of $C_i$ has an
$\F_v$-rational divisor of degree 1. By Hensel's lemma, it lifts to a $k_v$-rational divisor of degree $1$. Thus
$C_i$ has a $k_v$-rational divisor of degree $g-1$, so 
$\e(C_i/k_v)=0$.  For $v\notin S$ and any $t\in k^{\times}$, we thus have
$$(\Delta_i,t)_{k_v}=(-1)^{\e(C_i^t/k_v)+\dim_{\F_2}\big(J_i(k_v)/N(J_i(L_w))\big)}.$$

By the product formula 
$\prod_v (\Delta_i,t)_{k_v}=1$, the parity of 
$\dim_{\F_2}\Sel^2(J_i^t)-\dim_{\F_2}\Sel^2(J_i)$ is
determined by
$$\prod_{v\notin S}(-1)^{\e(C_i^t/k_v)}\cdot\prod_{v\in S}(\Delta_i,t)_{k_v}
\cdot
\prod_{v\in S}(-1)^{\dim_{\F_2}\big(J_i(k_v)/N(J_i(L_w))\big)}\in\{\pm1\}.$$
If $t\in\mathfrak S$, then $C_i^t$ is everywhere locally soluble, so the first product is 1.
If, moreover, $t$ is a unit at $w_1$ and $w_2$, then since $\val_{w_i}(\Delta_j)$ is even
by assumption, we have $(\Delta_i,t)_{k_{w_1}}=(\Delta_i,t)_{k_{w_2}}=1$, hence for $i=1,2$
the parity of 
$\dim_{\F_2}\Sel^2(J_i^t)-\dim_{\F_2}\Sel^2(J_i)$ is 
determined by
\begin{equation}
\prod_{v\in S\setminus\{w_1,w_2\}}(\Delta_i,t)_{k_v}\cdot
\prod_{v\in S}(-1)^{\dim_{\F_2}\big(J_i(k_v)/N(J_i(L_w))\big)} \in\{\pm1\}.\label{parity1}
\end{equation}
This holds for any $t \in \mathfrak{S}$ that is a unit at $w_1$ and $w_2$. 

By \Cref{prop:fibration-method}, the assumption
$X({\bf A}_k)\neq\emptyset$ implies that $\mathfrak S\neq\emptyset$.
For every $v\in S\setminus\{w_1,w_2\}$, choose
$t_v\in k_v^\times$ such that
$C_i^{t_v}(k_v)\neq\emptyset$ for $i=1,2$.
Then, using Lemma \ref{lemma4}, choose units $t_{w_1}\in k_{w_1}^\times$ and $t_{w_2}\in k_{w_2}^\times$ such that 
\begin{equation}
\prod_{v\in S\setminus\{w_1,w_2\}}(\Delta_i,t_v)_{k_v}\cdot
\prod_{v\in S}(-1)^{\dim_{\F_2}\big(J_i(k_v)/N(J_i(k_v(\sqrt{t_v})))\big)}\label{parity11}
\end{equation}
takes any given value in $\{\pm 1\}$, for $i=1,2$.
We note that $C_i^{t_{w_j}}(k_{w_j})\neq\emptyset$ for all
$i,j=1,2$ by Lemma \ref{lem:weil-bound} and Hensel's lemma.
By \Cref{prop:fibration-method}, there exists $t \in \mathfrak{S}$ sufficiently close to $t_v$ for all $v\in S$. 
Then the value of (\ref{parity1})
equals the value of (\ref{parity11}),
which we can choose to be such that
$\dim_{\F_2}\Sel^2(J_i^t)$ is odd for $i=1,2$.
\end{proof}

\section{Simultaneous variation of the Cassels–Tate pairing}
\label{app:simultaneous-morgan}
\label{S5}

\begin{defin}\label{def:morgan-datum}
Let $k$ be a number field, let $J$ be the Jacobian of a smooth projective
hyperelliptic curve
over $k$, let $L=k(J[2])$, and set $G=\Gal(L/k)$. Let $\w$ be the class
defined in Proposition \ref{prop:ps-class}(1), and suppose that 
\begin{equation}\label{eq:selmer-intersection}
    \on{Sel}^2(J)\cap H^1(G,J[2])=\ang{\w} \quad\text{as subgroups of $H^1(k,J[2])$.}
\end{equation}
Let $m:=\dim_{\F_2}(\ang{\w})\in\{0,1\}$ and
$n:=\dim_{\F_2}\Sel^2(J)/\ang{\w}$.
Thus we have $n+m=\dim_{\F_2}\Sel^2(J)$. Let
\begin{equation}\label{eq:S-morgan}
\begin{aligned}
\mathcal S&=\mathcal S_0\sqcup\mathcal S_1,\\
\mathcal S_0&=
\begin{cases}
\emptyset, & \text{if } m=0,\\
\{n+1\}\times\{1,\ldots,n\}, & \text{if } m=1,
\end{cases}
\qquad
\mathcal S_1=\{(r,q):1\leq r<q\leq n\}.
\end{aligned}
\end{equation}
By definition, a \emph{Morgan datum} for $J$ is a tuple
\begin{equation}\label{eq:morgan-datum}
(T=\{\mathfrak a_i\}_{i=1}^n,T'=\{\mathfrak a_i\}_{i={n+1}}^{n+m}, \{a_i\}_{i=1}^{n+m},\{\gamma_\nu\}_\nu)
\end{equation}
with the following properties:
\begin{enumerate}
    \item $T\sqcup T'=\{\mathfrak a_i\}_{i=1}^{n+m}$ is a basis of $\on{Sel}^2(J)$,
    \item the image of $T=\{\mathfrak a_i\}_{i=1}^n$ in $\on{Sel}^2(J)/\ang{\w}$ is a basis,
    \item $T'=\{\mathfrak a_i\}_{i={n+1}}^{n+m}$ is a basis of $\ang{\w}\subset \on{Sel}^2(J)$,
    \item for each $i\in\{1,\ldots,n+m\}$, the map $a_i\colon \Gamma_k\to J[2]$ is a cocycle representing $\mathfrak a_i$, which we require to be the inflation of a cocycle $\tilde{a}_i\colon G\to J[2]$ when $i>n$,
    \item for each $\nu=(r,q)\in\mathcal S$, the map $\gamma_\nu\colon \Gamma_k\to\mu_2$ is a
$1$-cochain such that $d\gamma_\nu=a_r\cup a_q$, where the cup-product is
formed using the Weil pairing on $J[2]$; such a cochain $\gamma_\nu$ exists because $\mathfrak a_r\cup \mathfrak a_q=0$ in $H^2(k,\mu_2)$ by
\cite[Remark~3.4]{Mor25}.
\end{enumerate}
\end{defin}

\begin{construction}\label{constr:morgan-field}
Let $k$ be a number field, let $J$ be the Jacobian of a smooth projective hyperelliptic curve
over $k$, let $L=k(J[2])$, and set $G=\Gal(L/k)$. Let $\w$ be the class
defined in Proposition \ref{prop:ps-class}(1), and suppose that \eqref{eq:selmer-intersection} holds. Fix a Morgan datum \eqref{eq:morgan-datum} for $J$. We attach to this datum a finite Galois extension $M/k$ and functions $\psi_\nu$ for $\nu\in\mathcal S$.

For every $\tau\in\Gamma_k$, let $\bar\tau$ denote its image in
$G=\Gal(k(J[2])/k)$. The cocycles $a_1,\ldots,a_n$ define a homomorphism
\[
\varphi_T\colon \Gamma_k\longrightarrow J[2]^n\rtimes G,\qquad
\tau\longmapsto ((a_i(\tau))_{1\leq i\leq n},\bar\tau),
\]
where $G$ acts diagonally on $J[2]^n$. Set
\[
\Gamma\coloneqq J[2]^n\rtimes G.
\]
For $1\leq i\leq n$, let $f_i$ be the $i$th coordinate cocycle on
$\Gamma$. For $n<i\leq n+m$, let $f_i\colon\Gamma\to J[2]$ be the inflation of
the cocycle $\tilde a_i\colon G\to J[2]$ of \Cref{def:morgan-datum}(4) along the projection $\Gamma\to G$. For every $1\leq i\leq n+m$, we have $a_i=f_i\circ\varphi_T$.

Following \cite[Section~4.3 and Notation~4.12]{Mor25}, let $E$ be the central
extension of $\Gamma$ by $\mu_2^{\mathcal S}$ associated with the $2$-cocycle
$(f_r\cup f_q)_{\nu=(r,q)\in\mathcal S}$. Since
$d\gamma_\nu=a_r\cup a_q$ for every $\nu=(r,q)\in\mathcal S$, the cochains
$\gamma_\nu$ give a lift of $\varphi_T$ to a homomorphism
\[
\varphi_{T,T'}\colon \Gamma_k\longrightarrow E.
\]
Explicitly, this lift is characterized by the requirement that its projection
to $\Gamma$ is $\varphi_T$, and that its $\nu$-th $\mu_2$-coordinate is
$\gamma_\nu$.

We define $M=k_{T,T'}$ to be the fixed field of $\ker(\varphi_{T,T'})$.
Thus $\varphi_{T,T'}$ identifies $\Gal(M/k)$ with a subgroup of $E$. In what
follows, for every $\sigma\in\Gal(M/k)$ we write $a_i(\sigma)$ and
$\gamma_\nu(\sigma)$ for $a_i(\tilde{\sigma})$ and $\gamma_\nu(\tilde{\sigma})$, where $\tilde{\sigma}\in \Gamma_k$ is a lift of $\sigma$. These values are independent of the choice of $\tilde\sigma$, since $\ker(\varphi_{T,T'})$ is contained in the kernels of all the maps $a_i$ and $\gamma_\nu$.

Let $\sigma\in\Gal(M/k)$, and suppose that the image of $\sigma$ in $G$ acts
on $J[2]$ without non-zero fixed points. It follows that the endomorphism $\sigma-1$ of $J[2]$ is invertible and hence, for each $1\leq i\leq n+m$, there is a unique element
$P_{i,\sigma}\in J[2]$ such that
\[
a_i(\sigma)=\sigma P_{i,\sigma}-P_{i,\sigma}.
\]
We write $\mu_2$ additively and identify it with $\F_2$. For
$\nu=(r,q)\in\mathcal S$, define
\[
\psi_\nu(\sigma)=e_2(P_{r,\sigma},a_q(\sigma))+\gamma_\nu(\sigma)\in\F_2,
\]
where $e_2\colon J[2]\times J[2]\to\mu_2$ is the Weil pairing associated with the
chosen principal polarization.
\end{construction}

\begin{lemma}
\label{lem:morgan-data}
Let $k$ be a number field, let $J$ be the Jacobian of a smooth projective hyperelliptic curve
over $k$, let $L=k(J[2])$, and set $G=\Gal(L/k)$. Let $\w$ be the class
defined in Proposition \ref{prop:ps-class}(1). Suppose that \eqref{eq:selmer-intersection} holds, that is, \[\on{Sel}^2(J)\cap H^1(G,J[2])=\ang{\w} \quad\text{as subgroups of $H^1(k,J[2])$}.\] 
Fix the sets $\mc{S}=\mc{S}_0\sqcup\mc{S}_1$ of \eqref{eq:S-morgan}, fix a Morgan datum for $J$ (in the sense of \Cref{def:morgan-datum}), and let $M=k_{T,T'}$ and $\{\psi_\nu\}_{\nu\in\mc{S}}$ be as in \Cref{constr:morgan-field}. Assume, moreover, that $J$ satisfies \cite[Assumption~1.3(A),(B)]{Mor25}: 
\begin{itemize} 
\item[(A)] $J[2]$ is a simple $\F_2[G]$-module and $\on{End}_G(J[2])=\F_2$; 
\item[(B)] the set $\{g\in G: J[2]^g=0\}$ is not empty. 
\end{itemize} 

Then the following hold.

\begin{enumerate}[label=\textup{(\alph*)}]
\item The field extension $M/L$ is a $2$-extension, that is, a finite Galois extension whose Galois group is a $2$-group. Moreover, every quadratic character $\Gal(M/k)\to \mu_2$ factors through $G$.

\item Let $g\in G$ satisfy $J[2]^g=0$. For every choice of values $\epsilon_\nu\in\F_2$, where $\nu\in\mathcal S$, there exists $\sigma\in\Gal(M/k)$ whose image in $G$ is $g$, and such
that $\psi_\nu(\sigma)=\epsilon_\nu$ for every $\nu\in\mathcal S$.

\item Let $\Sigma_0$ be a finite set of places of $k$ containing all
places above $2$, all archimedean places, all places of bad reduction for
$J$, and all places ramified in $M/k$. Let $\chi\colon \Gamma_k\to\mu_2$ be a quadratic character such that $\chi_v$ is trivial for every $v\in\Sigma_0$, and such that, for every $v\notin\Sigma_0$ at which $\chi$ is ramified, one has $J[2]^{\on{Frob}_v}=0$.
Then
\[
\Sel^2(J^\chi/k)=\Sel^2(J/k)\quad \text{as subgroups of $H^1(k,J[2])$},
\]
and
\[
\rk_2(J^\chi/k)\equiv\rk_2(J/k)\pmod 2.
\]

\item Under the hypotheses of \textup{(c)}, for all
$\nu=(r,q)\in\mathcal S$, one has
\[
\langle \mathfrak a_r,\mathfrak a_q\rangle_{J^\chi}
=
\langle \mathfrak a_r,\mathfrak a_q\rangle_J
+
\sum_{v\in\operatorname{Ram}(\chi)\setminus\Sigma_0}
\psi_\nu(\on{Frob}_v).
\]
\end{enumerate}
\end{lemma}

\begin{proof}
Part \textup{(a)} follows from \cite[Proposition 4.14]{Mor25}. Part
\textup{(b)} is \cite[Corollary 4.19]{Mor25}. Part \textup{(c)} is
\cite[Lemma 2.8]{Mor25}. Part \textup{(d)} is
\cite[Proposition 3.3]{Mor25}.
\end{proof}

\begin{lemma}
\label{lem:morgan-global-twisting}
Let $M/k$ be a finite Galois extension, and let $\Sigma$ be a finite set
of places of $k$, containing all archimedean places, all places above $2$,
and all places ramified in $M/k$. Let $\mathcal C\subset \Gal(M/k)$ 
be a nonempty union of conjugacy classes. For every $v\in\Sigma$, let
$\chi_v\colon G_{k_v}\to\mu_2$ be a quadratic character. Suppose that, for every quadratic character $\beta\colon \Gal(M/k)\to\mu_2$ which is trivial on $\mathcal C$, one has
\begin{equation}\label{eq:reciprocity-condition}
\sum_{v\in\Sigma}
\operatorname{inv}_v(\chi_v\cup\operatorname{res}_v\beta)=0.
\end{equation}
Then there exists a global quadratic character $\chi\colon \Gamma_k\to\mu_2$ such that $
\operatorname{res}_v(\chi)=\chi_v$ for every $v\in\Sigma$ and such that every place $v\notin\Sigma$ at which $\chi$ is ramified satisfies $\on{Frob}_v\in\mathcal C$. 
\end{lemma}

\begin{proof}
This is \cite[Proposition 5.2]{Mor25}.
\end{proof}

\begin{lemma}\label{lem:disjoint-auxiliary-fields}
For $i=1,2$, let $C_i:y^2=f_i(x)$ be hyperelliptic curves over $k$, with
$f_i$ separable of degree $2g_i+2\geq 6$, and let $J_i$ be their Jacobians.
Let $L_i$ be the splitting field of $f_i$, and assume that
$G_i\coloneqq\Gal(L_i/k)\simeq S_{2g_i+2}$. Assume that $L_1$ and $L_2$ are
linearly disjoint over $k$.

Let $t\in k^\times$. For $i=1,2$, we have the canonical identification 
of $\Ga_k$-modules $J_i^t[2]\simeq J_i[2]$; in particular $L_i=k(J_i^t[2])$. For $i=1,2$, assume
that \eqref{eq:selmer-intersection} holds for $J_i^t$, fix a Morgan datum for $J_i^t$, and let $M_i$ be the finite Galois extension of $k$
attached to this datum by \Cref{constr:morgan-field}. Then $M_1$ and $M_2$
are linearly disjoint over $k$. Moreover, every quadratic character
$\Gal(M_1M_2/k)\to\mu_2$ factors through $\Gal(L_1L_2/k)\simeq G_1\times G_2$.
\end{lemma}

\begin{proof}
Since $G_i\simeq S_{2g_i+2}$, with $2g_i+2\geq 6$, 
the hypotheses (A) and (B) of \Cref{lem:morgan-data} hold for $J_i^t$. Thus
\Cref{lem:morgan-data}\textup{(a)}, applied to $J_i^t$, shows that $M_i/L_i$
is a $2$-extension, and that every quadratic subextension of $M_i/k$
is contained in $L_i$.

We first show that $L_1\cap M_2=k$ as subfields of $\cl{k}$.  Set $F=L_1\cap M_2$; since $L_1/k$ and $M_2/k$ are Galois, $F/k$ is also Galois. Since $L_1$ and
$L_2$ are linearly disjoint over $k$ and $F\subset L_1$, we have $F\cap L_2=k$. The compositum $FL_2$ is contained in $M_2$, and $M_2/L_2$ is a $2$-extension. Hence
$FL_2/L_2$ is a $2$-extension. Since $F\cap L_2=k$ and $F/k$ is Galois, it follows that $F/k$ is a $2$-extension. On the other hand, $F\subset L_1$ and $\Gal(L_1/k)\simeq S_{2g_1+2}$. Since
$2g_1+2\geq 6$, the only nontrivial $2$-group quotient of $S_{2g_1+2}$ is its
quotient of order $2$, corresponding to the sign character. Therefore, if $F\neq k$,
then by Galois theory $F$ contains the quadratic subfield of $L_1$. But $F\subset M_2$, and
every quadratic subextension of $M_2/k$ is contained in $L_2$, contradicting
$L_1\cap L_2=k$. Hence $L_1\cap M_2=k$. The same argument gives
$L_2\cap M_1=k$.

Now set $F'=M_1\cap M_2$. Since
$F'\cap L_1\subset M_2\cap L_1=k$, the fields $F'$ and $L_1$ are linearly
disjoint over $k$. Moreover, $F'/k$ is Galois, since both $M_1/k$ and $M_2/k$
are Galois. The compositum $F'L_1$ is contained in $M_1$, and
$M_1/L_1$ is a finite Galois extension whose Galois group is a $2$-group. Since
$F'\cap L_1=k$, restriction gives an isomorphism
$\Gal(F'L_1/L_1)\simeq\Gal(F'/k)$. Hence $F'/k$ is a $2$-extension.

If $F'\neq k$, then $F'/k$ has a quadratic subextension $K/k$. Since
$K\subset M_1$ and $K\subset M_2$, \Cref{lem:morgan-data}\textup{(a)} gives
$K\subset L_1$ and $K\subset L_2$, contradicting $L_1\cap L_2=k$. Therefore
$M_1\cap M_2=k$. This shows that the Galois extensions $M_1/k$ and $M_2/k$ are linearly disjoint over $k$. In particular, $\Gal(M_1M_2/k)\simeq \Gal(M_1/k)\times\Gal(M_2/k)$.
By \Cref{lem:morgan-data}\textup{(a)} every quadratic character of $\Gal(M_i/k)$ factors through $G_i$, for $i=1,2$. It follows that every quadratic character of
$\Gal(M_1M_2/k)$ factors through
$G_1\times G_2\simeq\Gal(L_1L_2/k)$, as desired.
\end{proof}

\begin{defin}\label{def:admissible-pairing} Let $A$ be a principally polarized abelian variety over $k$, and let $\c\in\Sel^2(A)$ be the Poonen--Stoll class attached to the chosen principal polarization. We identify $\frac{1}{2}\Z/\Z$ with $\F_2$. Following Morgan \cite[\S~1.3]{Mor25}, a symmetric bilinear form \[ P\colon \Sel^2(A)\times\Sel^2(A)\longrightarrow\F_2 \] is called \emph{admissible} if $P(\c,x)=P(x,x)$ for all $x\in\Sel^2(A)$, and \[ P(\c,\c)\equiv \dim_{\F_2}\Sel^2(A)+\rk_2(A)+\dim_{\F_2}A(k)[2]\pmod 2. \] \end{defin}

The following elementary observation is implicit in Morgan's proof of
\cite[Theorem~1.14]{Mor25}.

\begin{lemma}\label{lem:admissible-determined-by-S}
Let $k$ be a number field, let $J/k$ and $\mc{S}=\mc{S}_0\sqcup \mc{S}_1$ be as in \Cref{def:morgan-datum}, and fix a Morgan datum \eqref{eq:morgan-datum} for $J$.
Let $P,Q\colon \Sel^2(J)\times\Sel^2(J)\to\F_2$ be admissible pairings. Suppose that $P(\mathfrak a_r,\mathfrak a_q)=Q(\mathfrak a_r,\mathfrak a_q)$ for every $\nu=(r,q)\in\mathcal S$. Then $P=Q$.
\end{lemma}

\begin{proof}
It is enough to show that the values of $P$ on the basis
$T\sqcup T'=\{\mathfrak a_i\}_{i=1}^{n+m}$ are determined by its values on the
pairs indexed by $\mathcal S$.

Consider the case where $g$ is even. Then by Proposition \ref{prop:ps-class}(2) we have $\w=\c$.

Suppose first that $\w=\c=0$. Then $m=0$, so $T'=\emptyset$ and
$T=\{\mathfrak a_1,\ldots,\mathfrak a_n\}$ is a basis of $\Sel^2(J)$. The
values on pairs $\mathfrak a_r,\mathfrak a_q$ with $1\leq r<q\leq n$ are
precisely the values indexed by $\mathcal S=\mathcal S_1$. Moreover,
admissibility gives $P(x,x)=P(\c,x)=0$ for every $x\in\Sel^2(J)$, so all
diagonal values are determined. Thus $P$ is determined by its values on
$\mathcal S$ in this case.

Suppose now that $\w=\c\neq 0$. Then $m=1$ and, with our notation,
$\mathfrak a_{n+1}=\c$. The values on pairs $\mathfrak a_r,\mathfrak a_q$
with $1\leq r<q\leq n$ are precisely the values indexed by $\mathcal S_1$.
The values $P(\c,\mathfrak a_q)$, for $1\leq q\leq n$, are precisely the
values indexed by $\mathcal S_0$. These determine the
diagonal values $P(\mathfrak a_q,\mathfrak a_q)$, since
$P(\mathfrak a_q,\mathfrak a_q)=P(\c,\mathfrak a_q)$
by the admissibility of $P$. Finally
\[
P(\c,\c)\equiv \dim_{\F_2}\Sel^2(A)+\rk_2(A)\pmod 2
\]
is also fixed by the admissibility of $P$. Thus $P$ is determined by its values on $\mathcal S$ also in this case.

Now consider the case where $g$ is odd. By Proposition \ref{prop:ps-class}(2) we have $\c=0$, so
any admissible pairing is alternating.
In this case the proof is straightforward.
\end{proof}

The following proposition is a variant of \cite[Theorem 1.14]{Mor25}. 

\begin{prop}
\label{prop:simultaneous-morgan}
Let $k$ be a number field. For $i=1,2$, let $C_i:y^2=f_i(x)$ be hyperelliptic curves with $f_i$ separable of degree $2g_i+2\geq 6$, and let $J_i$ be their Jacobians. Assume that $G_i\coloneqq\Gal(f_i)\simeq S_{2g_i+2}$, and that the splitting fields $L_1$ and $L_2$ of $f_1$ and $f_2$ are linearly disjoint over $k$. Suppose given $t\in\mathfrak S$. For $i=1,2$, let \[P_i\colon\Sel^2(J_i^t/k)\times\Sel^2(J_i^t/k)\longrightarrow \F_2\] be an admissible pairing (as in \Cref{def:admissible-pairing}). 
Then there exists $s\in\mathfrak S$ such that, for $i=1,2$, one has $\Sel^2(J_i^s/k)=\Sel^2(J_i^t/k)$ inside $H^1(k,J_i[2])$, one has $\rk_2(J_i^s/k)\equiv\rk_2(J_i^t/k)\pmod 2$, and the Cassels--Tate pairing on $\Sel^2(J_i^s/k)$ is equal to $P_i$.
\end{prop}


\begin{proof}[Proof of \Cref{prop:simultaneous-morgan}]
We note that, for $i=1,2$, property \eqref{eq:selmer-intersection} holds:
\[
\Sel^2(J_i^t/k)\cap H^1(G_i,J_i[2])=\langle\w_i\rangle.
\]
This follows from $t\in \mathfrak S$
and Proposition \ref{lem:maximal-galois-image}(3).

Observe that the principally polarized abelian varieties $J_1^t$ and $J_2^t$
satisfy the assumptions of \Cref{lem:morgan-data}. Indeed, for $i=1,2$, we
have $J_i^t[2]\simeq J_i[2]$, so that $k(J_i^t[2])=L_i$, and hence
$G_i=\Gal(L_i/k)\simeq S_{2g_i+2}$. Thus Assumptions \textup{(A)} and
\textup{(B)} of \Cref{lem:morgan-data} hold for $J_i^t$, while
\eqref{eq:selmer-intersection} holds by hypothesis. For each $i=1,2$, choose
a Morgan datum $(T_i,T_i',\{a_{i,r}\}_r,\{\gamma_{i,\nu}\}_{\nu\in\mathcal S_i})$
for $J_i^t$, in the sense of \Cref{def:morgan-datum}, where $\mathcal S_i=(\mathcal S_i)_0\sqcup(\mathcal S_i)_1$ is the corresponding finite indexing set. Applying
\Cref{constr:morgan-field} we obtain, for $i=1,2$, a finite Galois extension
$M_i=k_{i,T_i,T_i'}/k$ and functions
\[
\psi_{i,\nu}\colon \{\sigma\in\Gal(M_i/k): J_i[2]^{\bar\sigma}=0\}
\longrightarrow \F_2,
\qquad \nu\in\mathcal S_i,
\]
where for every $\sigma\in \on{Gal}(M_i/k)$ we denote by $\bar\sigma$ the image of $\sigma$ in $G_i$.

Let $M=M_1M_2$. By \Cref{lem:disjoint-auxiliary-fields}, we have $\Gal(M/k)\simeq\Gal(M_1/k)\times\Gal(M_2/k)$. We shall write an element of $\Gal(M/k)$ as a pair $(\sigma_1,\sigma_2)$, with $\sigma_i\in\Gal(M_i/k)$.

For $i=1,2$, choose a cycle $h_i\in G_i\simeq S_{2g_i+2}$ of length $2g_i+1$. By \Cref{lem:maximal-galois-image-new}(1), one has $J_i[2]^{h_i}=0$. For $\nu=(r,q)\in\mathcal S_i$, define
\[
\delta_{i,\nu} \coloneqq 
\langle \mathfrak a_{i,r},\mathfrak a_{i,q} \rangle_{J_i^t}
+
P_i(\mathfrak a_{i,r},\mathfrak a_{i,q})
\in\F_2.
\]
By \Cref{lem:morgan-data}\textup{(b)}, for each $i=1,2$ we can choose elements $\sigma_i^*,\sigma_i^0\in\Gal(M_i/k)$, both mapping to $h_i\in G_i$, such that $\psi_{i,\nu}(\sigma_i^*)=\delta_{i,\nu}$ and $\psi_{i,\nu}(\sigma_i^0)=0$ for all $\nu\in\mathcal S_i$.
Define $\tau_1\coloneqq (\sigma_1^*,\sigma_2^0)$ and $\tau_2\coloneqq (\sigma_1^0,\sigma_2^*)$, viewed as elements of $\Gal(M/k)$.

Choose a finite set of places $\Sigma_0$ containing
\begin{itemize}
    \item[(i)] all archimedean places and all places above $2$,
    \item[(ii)] all places ramified in $M/k$,
    \item[(iii)]  all places of bad reduction for $C_1$ or $C_2$,
    \item[(iv)] every place $v$ such that $t\notin\O_v^\times$, and
    \item[(v)] all places $v$ for which the residue field is too small to guarantee a smooth $\F_v$-point on every smooth projective curve over $\F_v$ of the form $y^2=a f_i(x)$, with $a\in\F_v^\times$ and $i=1,2$.
\end{itemize}
The last condition excludes only finitely many places, by \Cref{lem:weil-bound}.

By the Chebotarev density theorem, we may choose distinct primes $\mathfrak p_1,\mathfrak p_2\notin\Sigma_0$ such that $\on{Frob}_{\mathfrak p_i}$ is conjugate to $\tau_i$ in $\Gal(M/k)$, for $i=1,2$. We define $\Sigma\coloneqq \Sigma_0\cup\{\mathfrak p_1,\mathfrak p_2\}$.

Define a subset $\mathcal C\subset\Gal(M/k)$ as follows. An element $\tau=(\rho_1,\rho_2)$ belongs to $\mathcal C$ if and only if, for $i=1,2$, the image of $\rho_i$ in $G_i$ is conjugate to $h_i$, and $\psi_{i,\nu}(\rho_i)=0$ for every $\nu\in\mathcal S_i$. The set $\mathcal C$ is nonempty, since it contains $(\sigma_1^0,\sigma_2^0)$. It is a union of conjugacy classes, because conjugation preserves the conjugacy class of the image in $G_i$, and the functions $\psi_{i,\nu}$ are constant on conjugacy classes; see \cite[Remark 3.2]{Mor25}.

We now prescribe local quadratic characters at the places of $\Sigma$. Set $\chi_v=1$ for every $v\in\Sigma_0$. For $j=1,2$, choose a nontrivial ramified quadratic character $\chi_{\mathfrak p_j}\colon G_{k_{\mathfrak p_j}}\to\mu_2$. We verify condition \eqref{eq:reciprocity-condition} in \Cref{lem:morgan-global-twisting}. Let $\beta\colon \Gal(M/k)\to\mu_2$ be a quadratic character which is trivial on $\mathcal C$. By \Cref{lem:disjoint-auxiliary-fields}, the character $\beta$ factors through $G_1\times G_2$. Since $\sigma_i^*,\sigma_i^0\in\Gal(M_i/k)$ map to $h_i$, for each $i=1,2$, the images of $\tau_1$ and $\tau_2$ in $G_1\times G_2$ are both equal to $(h_1,h_2)$. Since $\mathcal C$ contains an element with image $(h_1,h_2)$, and since $\beta$ is trivial on $\mathcal C$, we get $\beta(\tau_1)=\beta(\tau_2)=1\in \mu_2$. Equivalently, $\operatorname{res}_{\mathfrak p_j}\beta$ is an unramified quadratic character taking the value $1\in \mu_2$ on $\on{Frob}_{\mathfrak p_j}$ for $j=1,2$. Hence $\operatorname{res}_{\mathfrak p_j}\beta$ is trivial, and therefore $\operatorname{inv}_{\mathfrak p_j}(\chi_{\mathfrak p_j}\cup\operatorname{res}_{\mathfrak p_j}\beta)=0$ for $j=1,2$. At all places of $\Sigma_0$, the prescribed local character is trivial, so the corresponding local invariants are zero. Hence condition \eqref{eq:reciprocity-condition} holds:
\[
\sum_{v\in\Sigma}
\operatorname{inv}_v(\chi_v\cup\operatorname{res}_v\beta)=0.
\]
By \Cref{lem:morgan-global-twisting}, there exists a global quadratic character $\chi\colon \Gamma_k\to\mu_2$ such that $\operatorname{res}_v(\chi)=1$ for all $v\in\Sigma_0$, such that $\operatorname{res}_{\mathfrak p_j}(\chi)=\chi_{\mathfrak p_j}$ for $j=1,2$, and such that every place $v\notin\Sigma$ at which $\chi$ is ramified satisfies $\on{Frob}_v\in\mathcal C$. Let $u\in k^\times/k^{\times 2}$ be the square-class corresponding to $\chi$, and set $s=tu$.

We first prove the assertions about Selmer groups and $2^\infty$-Selmer ranks. Fix $i\in\{1,2\}$. The character $\chi$ is trivial at every place of $\Sigma_0$. If $v\notin\Sigma_0$ and $\chi$ is ramified at $v$, then either $v=\mathfrak p_1$, or $v=\mathfrak p_2$, or $v\notin\Sigma$ and $\on{Frob}_v\in\mathcal C$. In all cases, the image of $\on{Frob}_v$ in $G_i$ is conjugate to $h_i$, and therefore $J_i[2]^{\on{Frob}_v}=0$ by \Cref{lem:maximal-galois-image-new}(1). Applying \Cref{lem:morgan-data}\textup{(c)} to $A=J_i^t$, for $i=1,2$, we obtain $\Sel^2(J_i^s/k)=\Sel^2(J_i^t/k)$ inside $H^1(k,J_i[2])$, and $\rk_2(J_i^s/k)\equiv\rk_2(J_i^t/k)\pmod 2$.

We now compute the Cassels--Tate pairings. Fix $i\in\{1,2\}$ and $\nu=(r,q)\in\mathcal S_i$. By \Cref{lem:morgan-data}\textup{(d)}, we have
\[
\langle\mathfrak a_{i,r},\mathfrak a_{i,q}\rangle_{J_i^s}
=
\langle\mathfrak a_{i,r},\mathfrak a_{i,q}\rangle_{J_i^t}
+
\sum_{v\in\operatorname{Ram}(\chi)\setminus\Sigma_0}
\psi_{i,\nu}(\on{Frob}_v).
\]
Suppose first that $i=1$. The contribution at $\mathfrak p_1$ is $\psi_{1,\nu}(\tau_1)=\delta_{1,\nu}$, while the contribution at $\mathfrak p_2$ is $\psi_{1,\nu}(\tau_2)=0$. Every further ramified prime outside $\Sigma$ has Frobenius in $\mathcal C$, so its $\psi_{1,\nu}$-value is zero. Therefore $\langle\mathfrak a_{1,r},\mathfrak a_{1,q}\rangle_{J_1^s}=P_1(\mathfrak a_{1,r},\mathfrak a_{1,q})$. The same argument for $i=2$, with $\mathfrak p_1$ and $\mathfrak p_2$ interchanged, gives $\langle\mathfrak a_{2,r},\mathfrak a_{2,q}\rangle_{J_2^s}=P_2(\mathfrak a_{2,r},\mathfrak a_{2,q})$ for every $\nu=(r,q)\in\mathcal S_2$.

Both the Cassels--Tate pairing on $\Sel^2(J_i^s/k)$ and $P_i$ are admissible. By Lemma \ref{lem:admissible-determined-by-S} an admissible pairing is determined by its values on the pairs indexed by $\mathcal S_i$, hence the Cassels--Tate pairing on $\Sel^2(J_i^s/k)$ is equal to $P_i$.

It remains to prove that $s\in\mathfrak S$. Equivalently, we must show that $C_i^s(k_v)\neq\emptyset$ for every place $v$ of $k$ and for $i=1,2$.

Let $v\in\Sigma_0$. Since $\chi_v$ is trivial, $u$ is a square in $k_v^\times$. Hence $C_i^s\simeq C_i^t$ over $k_v$. Since $t\in\mathfrak S$, we have $C_i^t(k_v)\neq\emptyset$, and therefore $C_i^s(k_v)\neq\emptyset$.

Now let $v\notin\Sigma_0$ be a place where $\chi$ is ramified. Then either $v=\mathfrak p_1$, or $v=\mathfrak p_2$, or $v\notin\Sigma$ and $\on{Frob}_v\in\mathcal C$. In all cases, the image of $\on{Frob}_v$ in $G_i\simeq S_{2g_i+2}$ is a cycle of length $2g_i+1$. By \Cref{lem:maximal-galois-image-new}(2), $f_i$ has a root $\alpha\in k_v$. Thus the point $(\alpha,0)$ lies on $C_i^s:y^2=s f_i(x)$, so $C_i^s(k_v)\neq\emptyset$.

Finally suppose that $v\notin\Sigma_0$ and that $\chi$ is unramified at $v$. Since $v\notin\Sigma_0$, the curve $C_i$ has good reduction at $v$, the leading coefficient of $f_i$ is a unit at $v$, and $t\in\O_v^\times$. Since $\chi$ is unramified, $u$ is represented by a unit in $k_v^\times$. Thus $s=tu$ is represented by a unit at $v$, and $C_i^s$ has good reduction at $v$. By the choice of $\Sigma_0$, the special fibre of $C_i^s$ has a smooth $\F_v$-point. Hensel's lemma lifts this point to a $k_v$-point on $C_i^s$.

We have shown that $C_i^s(k_v)\neq\emptyset$ for every $v$ and for $i=1,2$. Hence $s\in\mathfrak S$, as desired.
\end{proof}

\section{Proof of Theorem~\ref{t1}}
\label{S6}

\begin{lemma}\label{lemma1}
Let $C$ be a smooth, projective, geometrically integral curve over a number field $k$.
If $\Pic^1(C_{k_v})\neq\emptyset$ for all places $v$ of $k$
and $\Pic^1_{C/k}(k)\neq\emptyset$, then $\Pic^1(C)\neq\emptyset$, that is,
$C$ has a zero-cycle of degree $1$.    
\end{lemma} 

\begin{proof} We have $\Pic_{C/k}(k)=\Pic(C_{\bar k})^{\Gamma_k}$. The standard
Hochschild--Serre spectral sequence gives rise to an exact sequence
$$0\to\Pic(C)\to \Pic(C_{\bar k})^{\Gamma_k}\to \Br(k)\to\Br(C).$$
By the Albert--Brauer--Hasse--Noether theorem, the natural map $\Br(k)\to\oplus_v\Br(k_v)$ is injective. We obtain that the map $\Br(k)\to\Br(C)$ 
is injective when $C$ has a zero-cycle of degree 1 over all completions of $k$.
This implies that every element of $\Pic(C_{\bar k})^{\Gamma_k}$ of degree 1 lifts to an element of 
$\Pic(C)$ of degree 1. 
\end{proof}

\begin{thm}\label{thm:c1-c2-zero-in-sha}
We place ourselves in \Cref{sit1}. Assume that $X({\bf A}_k)\neq\emptyset$ and that $\Sha(J_i^u)[2^\infty]$ is finite for every $u\in\mathfrak S$ and $i=1,2$. Then there exists $s\in\mathfrak S$ such that $(\w_1,\w_2)$ maps to zero in $\Sha(J_1^s\times_k J_2^s)$. In particular, $C_1^s\times_k C_2^s$ has a zero-cycle of degree $1$.
\end{thm}

\begin{proof}
By \Cref{prop:odd-odd}, there exists $t\in\mathfrak S$ such that both $\dim_{\F_2}\Sel^2(J_1^t)$ and $\dim_{\F_2}\Sel^2(J_2^t)$ are odd. Since $t\in\mathfrak S$, we have $\langle \c_i,\c_i\rangle_{J_i^t}=0$ by \cite[Corollary~12]{PS}. Hence, by \Cref{cor:ps-class-twists}, \[ \dim_{\F_2}\Sel^2(J_i^t)+\rk_2(J_i^t)\equiv 0\pmod 2. \] Thus, in this situation, a symmetric bilinear form $P\colon \Sel^2(J_i^t)\times\Sel^2(J_i^t)\to\F_2$ is admissible (in the sense of \Cref{def:admissible-pairing}) if and only if $P(\c_i,x)=P(x,x)$ for all $x\in\Sel^2(J_i^t)$, and $P(\c_i,\c_i)=0$.

Since $t\in\mathfrak S$,
we have $\w_i\in \Sel^2(J_i^t)$
for $i=1,2$ by Proposition \ref{prop:ps-class}(4).
Moreover, for $i=1,2$ we have $\w_i\neq 0$ by Proposition
\ref{lem:maximal-galois-image}(3).

Let us construct an admissible pairing on $\Sel^2(J_i^t)$.
Since $\dim\, \Sel^2(J_i^t)$ is odd, there is an even-dimensional 
$\F_2$-vector subspace $V\subset  \Sel^2(J_i^t)$
such that $\Sel^2(J_i^t)=\F_2\w_i\oplus V$. Choose an alternating non-degenerate pairing on $V$
and extend it to a pairing $P_i(x,y)$ on $\Sel^2(J_i^t)$ with kernel $\F_2\w_i$. The pairing $P_i(x,y)$
 is admissible, because $\c_i=\w_i$
 if $g_i$ is even,
 or $\c_i=0$ if $g_i$ is odd.

By \Cref{lem:maximal-galois-image}(3), applied to $C_i^t$,
and using \Cref{prop:ps-class}\textup{(7)}, we have
\[
\Sel^2(J_i^t/k)\cap H^1(G_i,J_i[2])=\langle\w_i\rangle .
\]
Thus the hypotheses of \Cref{prop:simultaneous-morgan} are satisfied for the two
admissible pairings $P_i$ just constructed. Therefore there exists a single
$s\in\mathfrak S$ such that, for $i=1,2$,
\[
\Sel^2(J_i^s/k)=\Sel^2(J_i^t/k)
\quad\text{inside } H^1(k,J_i[2]),
\]
and the Cassels--Tate pairing on $\Sel^2(J_i^s/k)$ is equal to $P_i$.
Thus $\dim_{\F_2}\Sel^2(J_i^s/k)$ is odd for $i=1,2$.

Let $q_i\colon \Sel^2(J_i^s)\to \Sha(J_i^s)[2]$ be the natural surjection. Since the
Cassels--Tate pairing on $\Sel^2(J_i^s)$ is the pullback of the
Cassels--Tate pairing on $\Sha(J_i^s)$, we have
\[
\ker(q_i)\subset \operatorname{rad}(P_i)=\F_2\w_i.
\]
Thus $\ker(q_i)$ is either $0$ or $\F_2\w_i$. 

Assume 
$\ker(q_i)=0$. Then
\[
\dim_{\F_2}\Sha(J_i^s)[2]=\dim_{\F_2}\Sel^2(J_i^s),
\]
which is odd.
Under the hypothesis that $\Sha(J_i^u)[2^\infty]$ is finite for every $u\in\mathfrak S$,  this implies  that 
$
\dim_{\F_2}\Sha(J_i^s)_{\rm nd}[2] 
$
is odd. 
By
\cite[Theorem 8, (1)$\iff$(3)]{PS} (see Proposition \ref{prop:ps-class}(5) above), 
$\dim_{\F_2}\Sha(J_i^s)_{\rm nd}[2]$
is odd if and only if
$\langle\mathfrak c_i,\mathfrak c_i\rangle_{J_i^s}\neq 0$.
This is a contradiction, since for $s \in \mathfrak S$, by \Cref{prop:ps-class}(6), 
we have
$\langle\mathfrak c_i,\mathfrak c_i\rangle_{J_i^s}=0$. 

Therefore $\ker(q_i)=\F_2\w_i$, and hence $\w_i$ maps to
zero in $\Sha(J_i^s)$. Since $\w_i$ maps to $[\Pic^1_{C_i^s/k}]$ in $H^1(k,J_i)$, this implies that $\Pic^1_{C_i^s/k}(k)\neq\emptyset$.  We conclude by \Cref{lemma1}.
\end{proof}

\begin{rmk}
    In contrast to \cite{Mor25}, we do not need to use a second descent argument. Indeed, $s\in \mathfrak S$ implies $\ang{c_i,c_i}_{J_i^s}=0$ (see \Cref{prop:ps-class}(6)), and hence 
$\dim_{\F_2}\Sha_{\rm nd}(J_i^s)[2]$
     is even (see \Cref{prop:ps-class}(5)).
\end{rmk}

\begin{proof}[Proof of Theorem~\ref{t1}]
Since $X({\bf A}_k)\neq\emptyset$, by \Cref{thm:c1-c2-zero-in-sha} there exists $s\in \mathfrak S$ such that $C_1^s\times_k C_2^s$ has a zero-cycle of degree $1$. Letting $Y$ be the blow-up of $C_1^s\times_k C_2^s$ at the fixed locus of the involution, we deduce that $Y$ has a zero-cycle of degree $1$. Since $X$ is birational to the quotient of $Y$ by the induced involution, we conclude that $X$ has a zero-cycle of degree $1$, as desired.
\end{proof}

\begin{rmk}\label{rmk:wa}
    The set of $k$-points $(a_0,\dots,a_{2g_1+2};b_0,\dots,b_{2g_2+2})$ of $\mathbb{A}^{2g_1+3}_k\times_k\mathbb{A}^{2g_2+3}_k$ such that \Cref{sit1} holds for the polynomials $f_1(x)=a_{2g_1+2}x^{2g_1+2}+\dots+a_0$ and $f_2(x)=b_{2g_2+2}x^{2g_2+2}+\dots+b_0$ satisfies weak approximation.

    Indeed, let $S_0$ be a set of places of $k$ containing the $2$-adic places, the archimedean places, and the non-archimedean places whose residue field has size $<4\max\{g_1,g_2\}^2$. Fix two distinct places $w_1$ and $w_2$ not in $S_0$, let $\mathfrak p_1$ and $\mathfrak p_2$ be the corresponding prime ideals, and let $\pi_1\in \mathfrak p_1$ and $\pi_2\in \mathfrak p_2$ be uniformizers.

    Let $h_1\in \O_k[x_1]$ and  $h_2\in \O_k[x_2]$ be such that, for $i=1,2$, the leading coefficient of $h_i$ is a local unit at $w_1$ and $w_2$, and $h_i(x_i)\equiv (x_i^2-\pi_i^2)q_i(x_i)\pmod {\mathfrak p_i^3}$, for some polynomial $q_i(x_i)\in \O_k[x_i]$ whose reduction modulo $\mathfrak p_i$ is separable and non-zero at $x_i=0$, and the reduction of $h_i(x_i)$ modulo $\mathfrak p_j$ is separable for $j\neq i$.

    By a theorem of Ekedahl \cite[Theorem 1.3]{Eke90}, the set of separable polynomials $f_1(x_1)$ and $f_2(x_2)$ with linearly disjoint splitting fields and such that, for $i=1,2$, the polynomial $f_i(x_i)$ has degree $2g_i+2$ and Galois group $S_{2g_i+2}$, and is sufficiently close to $h_i(x_i)$ at $w_1$ and $w_2$, satisfies weak approximation.
\end{rmk}

\appendix

\section{The surfaces in Situation~\ref{sit1} are of general type}

In this appendix we prove the following statement.

\begin{lemma}\label{lemma:X-general-type}
Let $k$ be a field of characteristic zero. For $i=1,2$, let $f_i(x)\in k[x]$ be a separable polynomial of degree $2g_i+2\geq 6$. Let $X$ be a smooth, projective,
geometrically integral variety
birational to the affine surface
\[
        y^2=f_1(x_1)f_2(x_2).
\]
Then $X$ is of general type.
\end{lemma}
\begin{proof}
We may assume that $k$ is algebraically closed. Let $B_1\subset \P^1$ and $B_2\subset \P^1$ be the reduced divisors defined by $f_1$ and $f_2$, respectively. For $i=1,2$, since $\deg(f_i)=2g_i+2$ and $f_i$ is separable, $B_i$ has degree $2g_i+2$. Let $Y\coloneqq \P^1_k\times_k\P^1_k$, and for $i=1,2$ let $\on{pr}_i\colon Y\to \P^1_k$ be the $i$th projection. Then $D\coloneqq \on{pr}_1^*B_1+\on{pr}_2^*B_2$ is a reduced divisor of bidegree $(2g_1+2,2g_2+2)$ on $Y$. In particular
\[
        \O_Y(D)\simeq \O_Y(2g_1+2,2g_2+2)
        \simeq \O_Y(g_1+1,g_2+1)^{\otimes 2}.
\]
Let $L=\O_Y(g_1+1,g_2+1)$, so that $\O_Y(D)\simeq L^{\otimes 2}$. The equation $y^2=f_1(x_1)f_2(x_2)$ defines, after compactification and normalization, the double cover $\pi\colon Z\to Y$ associated with the line bundle $L$ and the branch divisor $D$. Thus $X$ is birational to a resolution of singularities of $Z$.

By construction $Z=\Spec_Y(\O_Y\oplus L^{-1})$, where the multiplication on $\O_Y\oplus L^{-1}$ is determined by a section $s\in H^0(Y,L^{\otimes 2})$ with zero locus $D$. It follows that $Z$ is, Zariski-locally on $Y$, a hypersurface inside the smooth variety $\mathbb{A}^1_k\times_kY$. In particular, $Z$ is Gorenstein, and hence the canonical sheaf $\omega_Z$ is invertible. By the canonical bundle formula for double covers,
we have
\[
        \omega_Z\simeq \pi^*(\omega_Y\otimes L);
\]
see \cite[Chapter I, Lemma 17.1(iii)]{BHPV}. Since $\omega_Y\simeq \O_Y(-2,-2)$, we get
\[
        \omega_Z\simeq\pi^*\O_Y(g_1-1,g_2-1).
\]
Since $g_1,g_2\geq 2$, the line bundle $\O_Y(g_1-1,g_2-1)$ is ample. It follows that $\omega_Z$ is ample.

The divisor $D$ is a union of vertical and horizontal lines. Away from the crossing points of $D$, the double cover $\pi$ is smooth. Near a crossing point, one may choose local parameters $s_1,s_2$ on $Y$ such that $D$ is given by $s_1s_2=0$, and so $Z$ is locally given by $z^2=s_1s_2$. 

The minimal resolution of the singularity $z^2=st$ is obtained by blowing up
the origin. Let $A=\A^3_{s,t,z}$ and let $W\subset A$ be the hypersurface
$z^2=st$. If $b\colon \widetilde A\to A$ is the blow-up of the origin, $E$ is
the exceptional divisor, and $\widetilde W$ is the strict transform of $W$, then
$\omega_{\widetilde A}\simeq b^*\omega_A\otimes \O_{\widetilde A}(2E)$,
see \cite[Chapter II, Exercise 8.5(b)]{HartshorneAG}. On the other hand, an explicit computation shows
$\O_{\widetilde A}(b^*W)\simeq
\O_{\widetilde A}(\widetilde W+2E)$.
By the adjunction formula we have
\[
        \omega_W\simeq (\omega_A\otimes \O_A(W))|_W, 
        \qquad \omega_{\widetilde W}\simeq
        (\omega_{\widetilde A}\otimes
        \O_{\widetilde A}(\widetilde W))|_{\widetilde W}.
\]
Therefore
\[
\begin{aligned}
        \omega_{\widetilde W}
        &\simeq
        \bigl(b^*\omega_A\otimes \O_{\widetilde A}(2E)
        \otimes \O_{\widetilde A}(\widetilde W)\bigr)|_{\widetilde W} \\
        &\simeq
        \bigl(b^*\omega_A\otimes \O_{\widetilde A}(2E)
        \otimes \O_{\widetilde A}(b^*W-2E)\bigr)|_{\widetilde W} \\
        &\simeq
        \bigl(b^*(\omega_A\otimes \O_A(W))\bigr)|_{\widetilde W} \\
        &\simeq
        (b|_{\widetilde W})^*\omega_W.
\end{aligned}
\]
Now let $\rho\colon \widetilde Z\to Z$ be the minimal resolution. Since $Z$ is
Gorenstein and $\widetilde Z$ is smooth, the invertible sheaf $\omega_{\widetilde Z}\otimes \rho^*\omega_Z^{-1}$ 
is of the form $\O_{\widetilde Z}(R)$ for a Cartier divisor $R$ supported on
the exceptional locus of $\rho$. Writing
\[
        R=\sum_P a_P E_P,
\]
where $P$ runs through the singular points of $Z$ and $E_P$ is the exceptional
curve over the $A_1$-singularity $P$, the local calculation above gives $a_P=0$ for every $P$.
Hence $R=0$, and therefore
$\omega_{\widetilde Z}\simeq \rho^*\omega_Z$.
Since $\omega_Z$ is ample, its pullback $\rho^*\omega_Z$ is nef and big. Hence
$\omega_{\widetilde Z}$ is nef and big. Therefore $\widetilde Z$ has Kodaira
dimension $2$. Since $X$ is birational to $\widetilde Z$, the surface $X$ is of
general type.
\end{proof}

\end{document}